\documentclass[12pt]{article}

\usepackage{epsfig,amsmath,amsthm,amssymb,amsfonts,xypic,graphicx}
\usepackage[all]{xy}
\begin{document}

%\newcounter{mycount}{1}
\newtheorem{thm}{Theorem}[subsection]
%\newtheorem{thm}{Theorem}[mycount]
%\addtocounter{mycount}{1}
\newtheorem{lem}[thm]{Lemma}
\newtheorem{cor}[thm]{Corollary}
\newtheorem{prop}[thm]{Proposition}
\newtheorem{remark}[thm]{Remark}
\newtheorem{defn}[thm]{Definition}
\newtheorem{ex}[thm]{Example}
\newtheorem{conj}[thm]{Conjecture}
\newtheorem{prop-conj}[thm]{Proposition-Conjecture}
\newtheorem{prop-defn3}[thm]{Proposition-Definition3}
\newtheorem{prop-defn}[thm]{Proposition-Definition}
\newtheorem{defn1}[thm]{Supplementary Definition}
\newtheorem{defn2}[thm]{Definition2}
%\newabstract{abs}{Abstract}[section]
\newenvironment{ack}{Acknowledgements}

%\numberwithin{equation}{subsubsection}
%\renewcommand{\theequation}{\thesubsection.\themycount}
%\counterwithin{equation}{subsection} \\ needs package chngcntr
%\numberwithin{equation}{mycount}
%\addtocounter{mycount}{1}
\newcommand{\mc}{\mathcal}
\newcommand{\mb}{\mathbb}
\newcommand{\surj}{\twoheadrightarrow}
\newcommand{\inj}{\hookrightarrow}
\newcommand{\red}{{\rm red}}
\newcommand{\codim}{{\rm codim}}
\newcommand{\rank}{{\rm rank}}
\newcommand{\Pic}{{\rm Pic}}
\newcommand{\Div}{{\rm Div}}
\newcommand{\Hom}{{\rm Hom}}
\newcommand{\im}{{\rm im}}
\newcommand{\Sym}{{\rm Sym}}
\newcommand{\Ker}{{\rm Ker}}
\newcommand{\Spec}{{\rm Spec \,}}
\newcommand{\Sing}{{\rm Sing}}
\newcommand{\Char}{{\rm char}}
\newcommand{\Tr}{{\rm Tr}}
\newcommand{\Gal}{{\rm Gal}}
\newcommand{\Min}{{\rm Min \ }}
\newcommand{\Max}{{\rm Max \ }}
\newcommand{\Ext}{{\rm Ext }}
\newcommand{\End}{{\rm End }}
\newcommand{\Tor}{{\rm Tor }}
\newcommand{\sgn}{\text{sgn}}
\newcommand{\del}{\partial}
\newcommand{\eff}{{\operatorname{eff}}}

\newcommand{\sA}{{\mathcal A}}
\newcommand{\sB}{{\mathcal B}}
\newcommand{\sC}{{\mathcal C}}
\newcommand{\sD}{{\mathcal D}}
\newcommand{\sE}{{\mathcal E}}
\newcommand{\sF}{{\mathcal F}}
\newcommand{\sG}{{\mathcal G}}
\newcommand{\sH}{{\mathcal H}}
\newcommand{\sI}{{\mathcal I}}
\newcommand{\sJ}{{\mathcal J}}
\newcommand{\sK}{{\mathcal K}}
\newcommand{\sL}{{\mathcal L}}
\newcommand{\sM}{{\mathcal M}}
\newcommand{\sN}{{\mathcal N}}
\newcommand{\sO}{{\mathcal O}}
\newcommand{\sP}{{\mathcal P}}
\newcommand{\sQ}{{\mathcal Q}}
\newcommand{\sR}{{\mathcal R}}
\newcommand{\sS}{{\mathcal S}}
\newcommand{\sT}{{\mathcal T}}
\newcommand{\sU}{{\mathcal U}}
\newcommand{\sV}{{\mathcal V}}
\newcommand{\sW}{{\mathcal W}}
\newcommand{\sX}{{\mathcal X}}
\newcommand{\sY}{{\mathcal Y}}
\newcommand{\sZ}{{\mathcal Z}}

\newcommand{\fM}{{\frak M}}
\newcommand{\fN}{{\frak N}}
\newcommand{\fD}{{\frak D}}
\newcommand{\fH}{{\frak H}}
\newcommand{\fG}{{\frak G}}
\newcommand{\fP}{{\frak P}}
\newcommand{\fU}{{\frak U}}
\newcommand{\fg}{{\frak g}}
\newcommand{\fu}{{\frak u}}

\newcommand{\A}{{\Bbb A}}
\newcommand{\B}{{\Bbb B}}
\newcommand{\C}{{\Bbb C}}
\newcommand{\D}{{\Bbb D}}
\newcommand{\E}{{\Bbb E}}
\newcommand{\F}{{\Bbb F}}
\newcommand{\G}{{\Bbb G}}
\renewcommand{\H}{{\Bbb H}}
\newcommand{\I}{{\Bbb I}}
\newcommand{\J}{{\Bbb J}}
\newcommand{\M}{{\Bbb M}}
\newcommand{\N}{{\Bbb N}}
\renewcommand{\P}{{\Bbb P}}
\newcommand{\Q}{{\Bbb Q}}
\newcommand{\R}{{\Bbb R}}
\newcommand{\T}{{\Bbb T}}
\newcommand{\U}{{\Bbb U}}
\newcommand{\V}{{\Bbb V}}
\newcommand{\W}{{\Bbb W}}
\newcommand{\X}{{\Bbb X}}
\newcommand{\Y}{{\Bbb Y}}
\newcommand{\Z}{{\Bbb Z}}
\def\sAf{\sA_{E^{\eff}}}
\def\sHf{\sH_{E^{\eff}}}
\def\mycases#1#2{\left.\vcenter{\hbox{$#1$}\hbox{$#2$}}\right.}
\title{A Candidate for the Abelian Category of Mixed Elliptic Motives}  
\author{Owen Patashnick \quad o.patashnick@bristol.ac.uk} 
\date{\today}
\maketitle

% work partially supported by Heilbronn Research Fellowship, Israel, and U of C scholarship
\begin{abstract}

In this work, we suggest a definition for the category of mixed motives generated by the motive $h^1(E)$ for $E$ an elliptic curve without complex multiplication.  We then compute the cohomology of this category.  Modulo a strengthening of the Beilinson-Soul\'e conjecture, we show that the cohomology of our category agrees with the expected motivic cohomology groups.  Finally for each pure motive 
$(\Sym^{n}h^1(E))(-1)$ we construct families of nontrivial motives whose highest associated weight graded piece is $(\Sym^{n}h^1(E))(-1)$.  

\end{abstract}

%\tableofcontents

\section{Introduction}\label{intro:begin}

Although a category of motives for smooth projective varieties, called pure motives, defined via algebraic cycles modulo homological equivalence, has been understood since the 1960s, we are only now starting to understand the outlines of a larger category of mixed motives.  For instance, Voevodsky \cite{V}, Levine \cite{Le2}, and Hanamura have each independently constructed a {\it derived} category of mixed motives.
However, one can still ask for more; a description of the as yet hypothetical abelian heart of such a category (with respect to the appropriate $t$-structure), at least in the case of rational coefficients.  This appears to be difficult.  However, some progress has been made in understanding the abelian subcategory of mixed Tate motives (with rational coefficients).  For example, Levine \cite {Le} (see also Goncharov \cite{G2}) showed that the abelian category of mixed Tate motives could be constructed inside the derived category of mixed motives.  A different line of attack was started earlier by Bloch and Kriz (\cite{B2}, \cite{BK}) who, building on ideas of Beilinson and Deligne \cite{BeD}, explicitly construct 
%a pro-Lie Group $G(T)$ (to be precise, they construct its dual, 
a $\Q$-graded Hopf algebra 
%$\sH_{\fM(T)}$
$\sH_T$, define the category of mixed Tate motives $\fM(T)$ over a field $k$ as the category of finite dimensional $\Q$-graded co-representations of $\sH_T$, and show that this category satisfies the major properties that such a category should satisfy\footnote{ Some of these properties require a strengthening of the Beilinson-Soul\'e conjecture, which we know is true for number fields by the work of Borel, Beilinson, and others, but is still conjectural for an arbitrary field.}.

Given the relative success of this ground up approach to mixed Tate motives one can ask whether other useful categories of mixed motives can be constructed that eventually will be understood as subcategories of the full abelian category of mixed motives.

Fix once and for all an elliptic curve $E$ without complex multiplication defined over a number field $k$ ($[k:\Q] <\infty$).  Let $h^*(E)$ denote the motive of $E$.  For the purposes of this introduction the reader can  take $h^*(E)$ to denote the singular (betti) cohomology of $E$.  Since $h^1(E)$ is a two dimensional vector space, with a natural action of $GL_2$, we identify $h^1(E)$ with the standard representation and let $P(E)$ denote the category of representations of $GL(h^1(E))$.  It follows from the classical theory of such representations that the objects of $P(E)$ are all of the form $\Sym^nh^1(E)\otimes (\wedge^2h^1(E))^{\otimes (-m)}$, where $n$ is a non-negative integer and  $m$ an arbitrary\footnote{$\wedge^2h^1(E)^{-1}$ denotes the formal dual of $\wedge^2h^1(E)$, so $\wedge^2h^1(E)^{\otimes m}$ for negative $m$ denotes $(\wedge^2h^1(E))^{-1})^{\otimes (-m)}$} integer.
After canonically identifying $\wedge^2h^1(E)$ with the Tate object $\Q(-1)$, we note that objects of 
$P(E)$ are all of the form $\Sym^nh^1(E)(m)$, where $Y(m)$ denotes $Y\otimes \Q(-1)^{\otimes m}$ (we say that $Y$ has been {\it twisted by} $m$),  and $\Q(1)$ is the usual formal dual of $\Q(-1)$.  

%Note that, for every object $Y$, there exists an $m$ such that $Y(m)$ is the dual of $Y$ (in other words, $Y(m)\otimes Y\rightarrow \Q$).  

%Although I am not aware of a source for these conjectures
% check with Wildeshaus
Unfortunately the available literature on mixed motives is fairly poor in general at spelling out for the non-expert what a mixed motive is expected to be and what properties categories of same are expected to satisfy, although the reader is encouraged to consult \cite{BeD} pp.\ 106-107 for a reader-friendly description of the category of mixed Tate motives.  For the convenience of the reader, therefore, we spell out explicitly the major conjectural properties expected to be true for a category of mixed elliptic motives "generated by" $h^1(E)$, where $E$ is as above:

A category $\fM(E)$ of mixed elliptic motives for our chosen $E$ 
%an elliptic curve defined over a number field $k$ is expected to be 
is expected to be 
\begin{description}
\item[Tannakian:] a Tannakian category 
%$\fM(E,k):=\fM(E)$ (just use 
%$\fM(E)$ 
defined over $\Q$ such that 
\begin{description}
\item[simple objects] the simple objects of $\fM(E)$ are the objects $\Sym^nh^1(E)(m)$ of $P(E)$ (which are assumed to be pairwise non-isomorphic), and which satisfy
\item[Vanishing in negative degrees:] for $m-2n>0$, $$\Ext^1(Sym^nh^1(E)(m),\Q)=0\footnote{or equivalently, $\Ext^1(Sym^{n_1}h^1(E)(m_1),Sym^{n_2}h^1(E)(m_2))=0$ for $m_1-2n_1>m_2-2n_2$.}$$
%\end{description}
\item[{\it Remark}] Note that the above two conjectures are equivalent to saying that every object $M\in \fM(E)$ has a unique finite filtration $W$, indexed by the elements of $P(E)$, where each associated graded is a sum of simple objects, and morphisms are strictly compatible with $W$.

\item[{\it Remark}]  Since $\fM(E)$ is Tannakian, we have (\cite{S}) that $\fM(E)
%\cong Rep(G_E)
$ is categorically equivalent to a category $Rep(G_E)$ of representations of an algebraic group scheme $G_E$.  
Since $P(E)$ is a subcategory of $\fM(E)$, we have a morphism of 
%It follows from the above remark that we have an action of $P(E)$ on $\fM(E)$ and hence a map of 
their Tannakian groups $G_E\rightarrow GL(h^1(E))$ with kernel a pro-unipotent group $U_E$.   
The functor $X\to \oplus Gr_{weight}(X)$
from $\fM(E)$ to $P(E)$ determines a morphism of groups in the opposite
direction, which splits the initial map and hence realizes $G_E$ as a
semidirect product 
%of GL_2 and some (pro)unipotent group U_E
%Furthermore, $GL(h^1(E))$ is normal in $G_E$, hence 
%It follows from the fundamental theorem of Tannakian Categories that 
%From the above remark 
%it follows that 
$G_E=GL(h^1(E))\rtimes U_E.$
%where $U_E$ is pro-unipotent.  
Applying the $\log$ map to $U_E$ yields a pro-nilpotent Lie algebra $\sL_E$, graded by elements of $P(E)$.   Alternatively we can dualize the situation.  The dual of a group scheme is a Hopf algebra $\sH_E$, and the dual of $\sL_E$ is a (graded) Lie co-algebra $\sM$.  
(As originally observed in \cite{Su}, any Lie co-algebra is a differential algebra.  Indeed, if $\partial$ denotes the dual of the Lie bracket, the dual of the 
%the dual of the Lie bracket is a differential\footnote{The observation that a Lie co-algebra is a differential algebra was originally made in \cite{Su}.} $\partial$ (as the dual of the 
Jacobi identity is precisely the condition that $\partial^2=0.$) Thus, in order to define a category $\fM(E)\cong coRep(\sH_E)$ (equivalently $\fM(E)\cong coRep(\sM_E)$), it suffices to define the Hopf algebra $\sH_E$ (resp. the Lie co-algebra $\sM$).  
%$\sA_{\eff}, \sA_{E^{\eff}}, \sAf$ $\sAE$
%\begin{description}
%(the functor to assoc graders is exact?)
\item[Agreement with motivic cohomology:]  There exist isomorphisms 
\begin{eqnarray*}
	\Ext^1_{\fM(E)}(\Sym^nh^1(E)(-m),\Q)\cong CH^{n+m}(E^n,2m+n-1)_{sgn}\otimes \Q
		% \\
%	\Ext^i_{\widehat\fM(E)}(\Sym^nh^1(E)(-m),\Q)=0\text{  for } i > 1
\end{eqnarray*}
where $sgn$ denotes the sign-character eigenspace of the appropriate higher Chow group.

\item[{\it Remark:}] Note that the {\bf Agreement with motivic cohomology} conjecture together with the Beilinson-Soul\'e conjecture implies the {\bf Vanishing in negative degrees} conjecture.
\item[Hodge realization]: For each complex embedding $\sigma:k\hookrightarrow \C$, there exits an exact $\otimes$ functor from $\fM(E)$ to the category of $\Q$-mixed Hodge structures\footnote{A reminder for the expert reader: Beilinson has shown that the existence of one realization functor implies the existence of all other realization functors.} which is compatible with the {\bf Agreement with motivic cohomology} conjecture.

%The special elements are elements of the category of motives, hence
%comodules over \sH_E H_{...}. They will be subcomodules of H_{...} as
%a comodule over itself -- the (co)regular (co)representation.  This
%subcomodule is cogenerated by some embeded \Sym^nh^1(E)(-1) into
%H_{...}. For this we shall construct element in \Sym^nh^1(E)(-1)
%component of H_{...}, it belongs to the sum of tensor product of
%\cal Z's without \boxtimes  \Sym^*h^1(E)(?) .

\item[Special elements:] There should exist a projective system of subcomodules of $\sH_E$ (where $\sH_E$ here is thought of as a comodule over itself), where each term is  cogenerated by an element in the $\Sym^nh^1(E)(-m)$-component of $\sH_E$.  More precisely, for each $\Sym^nh^1(E)(-m)$ for $m\geq 1$, there should exist filtered elements $\sE(n,m)\in\sM$ which have highest weight-graded piece $\Sym^nh^1(E)(-m)$ and lowest\footnote{we can always choose elements of $\sM$ to have lowest weight graded piece $\Q$ by taking a suitable twist} weight-graded piece $\Q$, and such that $\del(\sE(n,m))=\sum c_{n',m'}\sE(n',m')$ where the sum is taken over all nonnegative $n'<n$ and $m'<m$, the $c_{n',m'}$ are appropriate integers (some of which may be zero), and $\del$ is the natural differential (the dual of the Lie bracket) on $\sM$.
% contain a projective system
\item[{\it Remark:}] These generators (via the Beilinson conjectures) are related to the special values $L(\Sym^nE,n+m)$ of the $L$-function of symmetric powers of $E$ at or beyond the critical strip.  In this paper we will restrict our attention to elements $\sE(n,1)$ which conjecturally generate $\Ext^1(\Sym^nh^1(E)(-1),\Q)$.  More general elements $\sE(n,m)$ for $m>1$ will be addressed in \cite{P3}.  The situation for values inside the critical strip is more complicated (and quite interesting) and will be addressed in a future paper.

\end{description}

\end{description}

\begin{description}
\item Furthermore, $\fM(E)$ is also expected to
\begin{description}
%\item ...
%
\item[Tate:] contain the category of mixed Tate motives $\fM(T)$ over $k$ as a full subcategory.

\end{description}

\end{description}

%Note that additional special elements of $\fN(E)$ may also exist 

%$\fM^{\eff}(E)$ and $\fM(E)$ instead of hat

In this paper we begin to show that such a category exists by proving that our candidate $\fM(E)$ satisfies a number of the above conjectural properties:

\begin{thm} \label{Tann cat thm} There exists an explicitly constructible Tannakian category $\fM(E)$ which satisfies the {\bf simple objects} conjecture.
%and which contains the category of mixed Tate motives ({\bf Tate}).
\end{thm}

As noted above, $\fM(E)$ is constructed by explicitly describing a Hopf algebra $\sH_E$ (and associated Lie co-algebra $\sM$), which in turn is determined by a (commutative graded) differential graded algebra (DGA) $\sA_E$.  Note that in the process of constructing $\fM(E)$, we also construct a full subcategory $\fM(E)^{\eff}$ of {\it effective} mixed elliptic motives, where the simple objects of $\fM(E)^{\eff}$ are $\Sym^nh^1(E)(-m)$ for non-negative $n$ {\bf and} $m$.  $\fM(E)^{\eff}$ is not quite Tannakian; it is in fact the category of representations of a monoid scheme whose formal completion is $G_E$.

\begin{thm} \label{cohomology result}  
%If $\sA_E$ is a $K(\pi,1)$ in the sense of Sullivan (\cite{Su}), 
If $\sA_E$ is cohomologically connected, then then $\fM(E)$ and $\fM(E)^{\eff}$ satisfy {\bf Vanishing in negative degrees}.
If furthermore the cohomology of $\sA_E$ is concentrated in degree one, then $\fM(E)$ and $\fM(E)^{\eff}$ also satisfy {\bf Agreement with motivic cohomology}.
%\begin{eqnarray*}
%	\Ext^1_{\widehat\fM(E)}(\Sym^nh^1(E)(-m),\Q)=CH^{n+m}(E^n,2m+n-1)\otimes Q)_{sgn} 
%	% \\
%%	\Ext^i_{\widehat\fM(E)}(\Sym^nh^1(E)(-m),\Q)=0\text{  for } i > 1
%\end{eqnarray*}
\end{thm}

The conditional statements of Theorem \ref{cohomology result} are strengthenings of the Beilinson-Soul\'e conjecture (see section \ref{hand wave} for more discussion of this point).  

\begin{thm} \label{special_elements} There exist explicitly computable projective systems of elements $\{\sE(g_1,...,g_n)\}$ in both $\fM(E)$ and $\fM(E)^{\eff}$, where \linebreak $\sE(g_1,...,g_n)\in \sM_{\Sym^nh^1(E)(-1)}$ (the $\Sym^nh^1(E)(-1)$- graded piece of $\sM_{\fM(E)}$), whose associated weight graded pieces are the pure motives 
${\Sym^{n}h^1(E)(-1)}$, $\Sym^{n-1}h^1(E)(-1)$, $\dots,  h^1(E)(-1), \Q(-1)$, $h^1(E)$, and $\Q$.  Here \linebreak $g_1,\ldots,g_n \in k(E)^*$ are rational functions whose divisors are pairwise disjointly supported, and pairwise disjoint from the identity element.   In particular, both $\fM(E)$ and $\fM(E)^{\eff}$ satisfy the {\bf Special elements} conjecture.
\end{thm} 

Note that in general, a particular $\sE(g_1,...,g_n)$ does not define an extension.  A mixed motive is a filtered object, hence is a more general object than a pure motive or even an extension of pure motives (this is a point that tends to be obscured when one just looks at motivic cohomology).  However, appropriate linear combinations $\sum \sE(g_1,...,g_n)$ of such elements do define extensions (elements of $\Ext^1_{\fM(E)}(\Sym^nh^1(E)(-1),\Q)$).

%The reason we care about such elements is the following special case of the Beilinson conjectures:
\begin{conj} \label{value generation}The $\sE(g_1,...,g_n)$ generate $\Ext^1_{\fM(E)}(\Sym^nh^1(E)(-1),\Q)$
\end{conj}
A special case of the Beilinson conjectures should then follow from a proof of conjecture \ref{value generation} and a suitable definition of a functor from $\fM(E)$ to the category of mixed Hodge structures; the image of $\Ext^1_{\fM(E)}(\Sym^nh^1(E)(-1),\Q)$ should be a lattice, and the volume of a fundamental domain of this lattice should be a known, explicit rational multiple of $L(\Sym^nE, n+1)$, the value of the $L$ function of a symmetric power of $E$ at $n+1$.  We will address conjecture \ref{value generation} as well as the {\bf Hodge realization} conjecture further in \cite{P2}.
%$L(Sym^nh_1(E),n+1)$

One reason to study an abelian category of motives (rather than relying on a derived category of motives) is that such a category is a finer tool for probing the structure of varieties, and its study elicits structures not detected by the derived category.  

The elliptic polylog mixed Hodge structures and their motivic analogues were first explored by Beilinson and Levin in \cite{BeL}, followed by Wildeshaus (\cite{W}, \cite{W2}, \cite{W3}) .  Their approach, while fairly explicit, is not particularly conducive to computation.  In an appendix (\cite{L}) to \cite{BeL},  Levin constructs explicit elements of the appropriate higher $K$-groups that realize the motivic elliptic polylog elements in \cite{BeL}.  The special elements we construct in section 4 (for Theorem \ref{special_elements} above) should be thought of as the algebraic cycle-theoretic analogues of the elements in \cite{L}.

This paper was written concurrently with \cite{B1}.  Both this paper and \cite{B1} grew out of work done by the author for the author's dissertation under the direction of S 
Bloch (as announced in \cite{B1}).  

Note that the main purpose of \cite{B1} is to relate the constructions defined in this paper with those in \cite{GL}.  In particular, the conditions (a) and (b) of Theorem 1.1, p.394 of \cite{GL} arise quite naturally in our context.

In \cite{G} (p. 25) A Goncharov cites \cite{B1} as a reference for a description of the elliptic motivic Lie co-algebra similar to the ones outlined in \cite{B1} and in section 2 of this paper. 
Note that preliminary drafts of \cite{G} did not have a description of this construction.  

The author was informed that both \cite{B1} and the parts of \cite{G} relevant to the construction outlined in section \ref{chap:DGA}  came about as a result of a discussion between S Bloch and A Goncharov regarding the author's dissertation at a conference in the fall of 1996.

The plan of the sections is as follows.  In section 2 we define the categories $\fM(E)$ and $\fM(E)^{\eff}$.  This proves Theorem \ref{Tann cat thm}.  In section 3 we prove Theorem \ref{cohomology result}.
%show that, if $\fM(E)$ is a $K(\pi,1)$ in the sense of Sullivan, that the Ext groups of our category agree with the expected motivic cohomology groups.  
%Let $\sH=h^1(E)(1)$.
In section 4 we define families of motives with weight graded pieces 
$\Sym^{n}h^1(E)(-1)$, $\Sym^{n-1}h^1(E)(-1)$,$\dots$, $h^1(E)(-1), \Q(-1), h^1(E)$, and $\Q$, which proves Theorem \ref{special_elements}.

%In \cite{G} (p. 25) Goncharov cites \cite{B1} as

\section{The Motivic DGA}\label{chap:DGA}

In this chapter we will define the categories $\fM(E)$ of mixed elliptic motives and the full subcategory $\fM(E)^{\eff}$ of effective mixed elliptic motives by explicitly constructing the motivic Hopf algebra $\sH_E$ and the motivic bi-algebra 
$\sHf$.  
For expository reasons it will be easier to start with the subcategory of effective mixed elliptic motives and then pass to the larger category of all mixed elliptic motives.

\subsection{Some Initial Notation}
We will assume that the reader is familiar with the concepts of a 
minimal model,
1-minimal model, and the bar construction for a commutative
differential graded algebra (DGA) 
$A$.  The concept of a generalized minimal model is due originally
to Quillen (see for example \cite{Q}).  In the form used in this paper 
(extensions by free 1 dimensional models)
it is due originally to Sullivan (\cite{Su}, see discussion starting p.\ 316).  
A good reference for the applications of minimal models we have in mind is 
the treatment in \cite{KM}, Part IV.  The bar
construction is due originally to Eilenberg and Mac Lane.  Good references for the use of the bar construction in this paper are \cite{C} and (\cite{BK}, section 2).

In the introduction we fixed once and for all an elliptic curve $E$ without complex multiplication over a number field $k$:
$$\pi:E\rightarrow \Spec(k)$$
Cycles with $\Q$-coefficients are to be understood for the 
remainder of the paper unless otherwise stated.  

Let $h: Var_{\Q} \rightarrow P$ be the functor that sends (smooth projective) varieties to the category of pure motives over $k$.  For every $n$ we have a motive $h(E^n)\in P(E) \subset P$, where $P(E)$ denotes the category of pure elliptic motives.  

%Let $$\sH:=h^1(E)(1)(=(R^1\pi_*\Q)(1))$$

Note that $\Q(-1)$ is the 
direct summand 
$\wedge^2h^1(E)$ of $h^1(E)\otimes h^1(E)$.

The tensor product of an object of an additive k-linear category
$\sC$ with
a (possibly infinite dimensional) vector space is defined
(representable-functorially) as
follows:  let $A\in Ob\sC$, $V$ a k-vector space.  Then for all
$B\in Ob\sC$, 
$$\text{Hom}(A\otimes V,B):=\text{Hom}(V,\text{Hom}(A,B))$$
Since we want to talk about tensor products of pure motives with
infinite dimensional vector spaces, we are really working in the
ind-category of pure motives.

Since $E$ is regular, we can (for the purposes of this paper) define motivic cohomology of $X=E\times E\times E\ldots\times E$ in terms of the cubical higher Chow groups:
$$
H^i_{\sM}(X,\Z(j)):= CH^j(X,2j-i).
$$
%\end{multline}
In this paper we will restrict to the study of rational motivic cohomology.

Let $\tilde{\sZ}^{a}(E^{b},\>\> \cdot \>\>)$ denote the terms in the cubical
higher Chow group complex $\otimes \Q$.  (Higher Chow groups are defined for example in \cite{B3}.  For the definition of cubical higher
Chow groups see for example Totaro \cite{T} pp.\ 178-180)

We will need to impose relations on the cubical higher
Chow group complex in order to give this complex the structure of a suitable commutative-graded differential graded algebra.

We define an action of $G_{c}:= (\Z/2\Z)^{c}\rtimes \Sigma_{c}$ on
$\tilde{\sZ}^{a}(E^{b},c)$ where $\Sigma_{c}$ acts on $(\P^1-\{1\})^{c}$ by permutation and $(\Z/2\Z)^{c}$ acts on 
$(\P^1-\{1\})^{c}$ with the $j$-th coordinate vector 
acting by $x\mapsto
x^{-1}$ on the $j$-th factor. Let 
$\text{Alt}_{G_c}=\sum_{\sigma \in G_{c}}(-1)^{sgn \>\> \sigma}\sigma$ denote the alternating projection with respect to this grading.  Then we define
$$\sZ^{a}(E^{b},c):=\text{Alt}_{G_c}(\tilde{\sZ}^{a}(E^{b},c)).$$

It follows (see for example \cite{B2}) that $\sum \sZ^{a}(E^{b},\> c\>)$ defines a graded-commutative DGA.

%We now intend to define a PEM-graded %category is filtered but not the algebra
% differential-graded algebra (DGA) $\sN$ which is then used to
%construct the Galois group of mixed elliptic motives.  

\subsection{Effective mixed elliptic motives: The DGA $\sAf$ (and the Motivic bi-Algebra $\sHf$)}\label{effective}

We define a left action of 
$$G_{b}:= (\Z/2\Z)^{b}\rtimes \Sigma_{b}$$ on 
$$\sZ^b_{c}:=\sZ^{b}(E^{b},c)
%(=\text{Alt}_{(\Z/2\Z)^{b}}(\sZ^{b}(E^{b},c))
$$
as follows:  $\Sigma_{b}$ acts by permuting the copies of 
$E$ in 
$E^{b}$ . The generators of $(\Z/2\Z)^b$ act
by $x\mapsto -x$ (in the group law on $E$) on the appropriate copy of $E$. 
We will think of this action as a right action as follows:  Let $\sigma\in\Q[\Sigma_b]$ be a generator.  Let $|\sigma|$ denote the number of transpositions in a minimal decomposition of $\sigma$.  Then 
$$\sZ^b_{c}\cdot \sigma:=((-1)^{{\rm signature }(\sigma)}(\sigma)^{-1})(\sZ^b_{c})=((-1)^{(|\sigma|+1)}(\sigma)^{-1})(\sZ^b_{c}).$$
We note for future reference that if $p\in\Q[\Sigma_b]$ is a projector 
%(coming from either a Young tableau or a Young tabloid) 
then 
$$\sZ^b_{c}\cdot p:=p^t(\sZ^b_{c})$$
where $p^t$ denotes the transpose of the projector $p$.

We also have a left action 
of $\Sigma_{b}$ 
by permutations 
on
$$h^1(E)^{\otimes b}
=:h^b
$$ 

\begin{defn}
$\sAf\>\>
%(\>\>=\sA\>\>)
:= \sum_i \sAf^i  := \sum_i\sum_b\sAf^i(b)$, where
\begin{eqnarray*}
\sAf^i(b)
&:=&\text{\rm Alt}_{(\Z/2\Z)^{b}}(\sZ^{b}({E}^{b},b-i))\otimes_{\Q[\Sigma_b]}h^1(E)^{\otimes b}\\
%&=:&\sZ^b_{b-i}\otimes_{\Q[\Sigma_b]}h^{b} 
%\big{(}
&=:&\sZ^b_{b-i}\boxtimes h^{b}
\end{eqnarray*}
\end{defn}

%{\bf Notation}$\sAf^i(b)
%=:\sZ^b_{b-i}\otimes_{\Q[\Sigma_b]}h^{b} 
%=:\sZ^b(E^b,b-i)\boxtimes h(E^{b})\quad 
%$
Note that $\sAf^i(b)$ is an (infinite-dimensional) $GL_2$-representation and that 
%For future reference we let $\boxtimes:=\otimes_{\Q[\Sigma_b]}$. 
%We define 
%\begin{defn}
%$$\sAf\>\>(\>\>=\sA\>\>):= \sum_i \sAf^i  := \sum_i\sum_b\sAf^i(b)$$
%\end{defn}
$\sAf$ lives in the ind-category of (infinite-dimensional) $GL_2$-representations.

We now define an algebra structure on $\sAf$.  
Notice that 
\begin{align}\label{squishies}
(\sZ^b_{b-i}\otimes_{\Q[\Sigma_{b}]}h^{b})\otimes(\sZ^{b'}_{b'-i'}\otimes_{\Q[\Sigma_{b'}]}h^{b'}) 
=(\sZ^b_{b-i}\otimes\sZ^{b'}_{b'-i'})\otimes_{\Q[\Sigma_{b}\times \Sigma_{b'}]}(h^{b}\otimes h^{b'}) 
\end{align}
since $$\Q[\Sigma_b]\otimes\Q[\Sigma_{b'}]=\Q[\Sigma_b\times \Sigma_{b'}].$$
We have external product maps on cycles 
$$\sZ^b_c\otimes_{\Q}\sZ^{b'}_{c'}\rightarrow \sZ^{b+b'}_{c+c'};\quad\quad
\sZ^{b'}_{c'}\otimes_{\Q}\sZ^b_c\rightarrow \sZ^{b+b'}_{c+c'}$$
and an external product on $GL_2$-representations 
$$h^b\otimes h^{b'}\rightarrow h^{b+b'}$$
which are compatible with the map 
$\Sigma_b\times \Sigma_{b'}\rightarrow \Sigma_{b+b'}$
(and hence the map \linebreak
$\Q[\Sigma_b\times \Sigma_{b'}]\rightarrow \Q[\Sigma_{b+b'}]$), i.e.
$$(C_1g_1)\otimes (C_2g_2)=(C_1\cdot C_2)(g_1\times g_2); \quad
(g_1h_1)\otimes (g_2h_2)=(g_1\times g_2)(h_1\cdot h_2).$$
This induces an algebra structure on 
$\sAf$ 
%%%%%(in fact, an Adams-graded DGA structure) 
under the
product map
\begin{equation}\label{1stproduct} 
\begin{array}{ccccc} 
%\begin{eqnarray*}
\sAf^i(b)\otimes_{\Q} \sAf^{i'}(b') 
&
\stackrel{\delta}{\longrightarrow}
&
 \sAf^{i+i'}(b+b');  \\ 
% \\
%\text{ i.e. } (\sZ^b_{b-i}\otimes\sZ^{b'}_{b'-i'})\otimes_{\Q[\Sigma_{b}\times \Sigma_{b'}]}(h^{b}\otimes h^{b'})&\longrightarrow& (\sZ^{b+b'}_{b+b'-i-i'}\otimes_{\Q[\Sigma_{b}\times \Sigma_{b'}]}(h^{b+b'}) \\ \\
%&\searrow^{\delta}&\downarrow \\
%&&(\sZ^{b+b'}_{b+b'-i-i'}\otimes_{\Q[\Sigma_{b+b'}]}(h^{b+b'});  \\ \\ 
\text{where } (C_1\boxtimes h_1)\otimes_{\Q} (C_2\boxtimes h_2)
&
\mapsto 
&
 (C_1\cdot C_2)\boxtimes(h_1\cdot h_2).
\end{array}
\end{equation}
%\end{eqnarray*}
%\end{multline}
%Indeed, it suffices to check that $\sL:=\sum \sL_b$ is well defined as a graded ideal with respect to this product.  This follows from the simple observation that if $a$ and $b$ are two disjoint permutations, then $ab=ba$.
%In other words, let $g_1,g_2, k \in \Q[\Sigma_b]$.  Then 
%\begin{eqnarray*}
%(Cg_1\otimes g_2h)(C'k\otimes h'-C'\otimes kh')&=&
%\big{(}(CC'g_1k\otimes g_2hh')-(CC'g_1\otimes g_2khh')\big{)} \\
%&=& \big{(}(CC'g_1k\otimes g_2hh')-(CC'g_1\otimes kg_2hh')\big{)} \\&&\quad \quad \quad \in \sL.
%\end{eqnarray*}

Note that $\sAf$ is not connected with respect to the grading by $i$.  However, (here we follow \cite{BK}) note that 
the {\it Adams grading}
$$\sAf= \sAf(0)\oplus \sAf(1)\oplus\cdots$$satisfies the following properties:
\begin{enumerate}
\item $\sAf(0)\cong k$,  $\sAf(b)=0$ for $b<0$ (in particular, $\sAf$ is connected with respect to the Adams grading),
\item The differential $\partial$ has Adams degree zero i.e., each $\sAf(b)$ is a subcomplex of $\sAf$,
\item The Adams grading is compatible with the algebraic structure i.e.,   
$\sAf(b)\otimes\sAf(b')\rightarrow \sAf(b+b')$.
\end{enumerate}

Since the Adams grading (codimension of cycle) gives $\sAf$ a connected graded structure, we can calculate the 1-minimal model ($\wedge \sM_{\sA}[-1]$) of $\sAf$ (in the sense of Sullivan - see for example \cite{Su}, discussion beginning p.\ 316).
Let $B(\sA)$ denote the bar construction of $\sA$.  

\begin{defn}
%Define 
$\sHf=H^0(B(\sAf^{\bullet}))$.  The associated Lie coalgebra
 $\sM$ is $I/I^2$ where $I$ denotes the augmentation ideal of $\sHf$.
%\end{defn}

%\begin{defn}
The category $\fM(E)^{\eff}$ of {\it effective} mixed elliptic motives is defined to be 
%\linebreak 
the category of 
comodules over $\sHf$, or (equivalently) the category of 
co-representations of the Lie co-algebra $\sM$. 
\end{defn}

This definition makes the most sense philosophically if the cohomology of $\sA_E$ is concentrated in degree one.
%$\sAf$ is a $K(\pi,1)$.  
We will say a bit more about the philosophy at the end of section 3.

We actually have a finer graded structure on our algebra than the one given by
%the pure motive structure is a finer structure than 
the Adams grading.  Let 
$\sZ^{b}:=\sum_i\sZ^b_{b-i}$.
Let $p \in \Q[\Sigma_b]$ be a projector (note that projectors are idempotents ($p^2=p$)).  
%For example, we can (though by no means do we have to) take $p=\varphi\otimes \text{Alt}$ where $\varphi\in \Q[\Sigma_b]$ is an idempotent corresponding to a
%Young symmetrizer.  (We could for example take $p=\phi\otimes \text{Alt}$ where $\phi\in \Q[\Sigma_b]$ is an idempotent corresponding to a tabloid of appropriate shape, or take $p$ to be defined by a linear combination of Young symmetrizers which happens to be idempotent).

Then
$$\sAf(b)\supset \sZ^b\boxtimes p \cdot h^{b}
=\sZ^b\boxtimes p^2\cdot h^{b}
=\sZ^b\cdot p \boxtimes p\cdot h^{b}$$
Note that this computation makes sense in the ind-category of pure elliptic motives.  ($\sZ^b_{b-i}$ is infinite dimensional, but is a direct sum of finite dimensional $GL_2$-representations.) 

Thus if $p_1$ and $p_2$ are two projectors, the product map (\ref{1stproduct})
on $\sAf$ induces a product map ($\star$):
\begin{align}\nonumber
(\sZ^b\cdot p_1 \boxtimes p_1\cdot h^{b})\otimes 
(\sZ^{b'}\cdot p_2 \boxtimes p_2\cdot h^{b'}) \stackrel{\delta}{\rightarrow}
(\sZ^{b+b'}\cdot(p_1\otimes p_2)  \boxtimes (p_1\otimes p_2)\cdot h^{b+b'})
\end{align}

As noted at the beginning of this section, the right action of a projector $p$ on the cycle side corresponds to a left action by $p^t$.
%and $V$ denotes $h^1(E)$ in even weight, since 
%$\Sigma_b$ acts as though 

For example, let $\Sym^b(h^1(E))$ denotes the trivial module, 
%$p^t\in \Q[\Sigma_b]$ denote the projection associated to the transpose of a Young symmetrizer, 
$\Sym_{\Sigma_b}$ the projection associated to the 
trivial projection $\sum_{\sigma\in \Sigma_b}\sigma$, and $\text{Alt}_{\Sigma_b}$ denotes the projection associated to the alternating projection $\sum_{\sigma\in \Sigma_b}(-1)^{|\sigma|+1}\sigma$.
Then 
\begin{eqnarray*}\sZ^b\boxtimes \text{Alt}_{\Sigma_b}\cdot h^b
%=\sZ^b\cdot\text{Alt}_{\Sigma_b}\boxtimes\Sym_{\Sigma_b}\cdot h^b
&=&\sZ^b\cdot\Sym_{\Sigma_b}\boxtimes\Sym^b(h^1(E)) \\
&=&\text{Alt}_{\Sigma_b}(\sZ^b)\boxtimes\Sym^b(h^1(E)) \\
&=&\Sym_{\Sigma_b}^t(\sZ^b)\boxtimes\Sym^b(h^1(E))
\end{eqnarray*}
since $(p^t)^t=p$.

%\vspace{.1in}

We will feel free at times to suppress writing projection on the cycle side when we write $(\text{cycle}) \boxtimes(\text{irreducible representation})$.

We are now in a position to understand the piece of $\sHf$ graded by 
the pure motive $p \cdot h^{b}$.

The proof of the following lemma is straightforward:

\begin{lem} The differential $\partial$ and the product map $\delta$ on cycles commute with the 
projection induced via an arbitrary Young symmetrizer.
\end{lem}
%\begin{proof}  
%Given a choice of coordinates on $E^n$ and a cycle $C\in \sZ^n(E^n, \> \cdot\>)$, $\partial$ restricts the support of $C$ but does not change the number or the order of the $E$-coordinates of $C$, while an arbitrary element $\sigma$ of $\Sigma_{n}$ changes the order of coordinates but does not change the number of coordinates or the support of $C$.  Thus $\partial$ 
%commutes with the action of $\sigma$.
%\end{proof}

%\begin{lem} The product map on cycles commutes with the 
%projection induced via an arbitrary Young symmetrizer.
%\end{lem}
%\begin{proof}  
%A Young symmetrizer $p$ cannot act on a product of cycles $C\otimes D, C\in \sZ^n(E^n, \> \cdot\>)$, $D\in \sZ^m(E^m, \> \cdot\>)$, unless 
%$p\in \Q[\Sigma_n\times\Sigma_m]$.  But then both the product map on cubical cycles and the symmetrizer $p$ preserve the relative order of the $E$-coordinates of $C\otimes D$.
%\end{proof}

%Furthermore, the product map commutes with the 
% projection induced via an arbitrary Young symmetrizer since after tensoring over $\Q[\Sigma_b]$ there is no longer any $\Sigma_b$ action.

In particular, the 
projection induced via an arbitrary Young symmetrizer commutes with the total differential on the bar complex.

It follows that our algebra is graded by irreducible representations of $GL_2$.
Indeed, 
%since every irreducible representation $\sU$ of $h^b$ is of the form $p \cdot h^b$ for some projector $p$,
we have a direct sum decomposition 
\begin{eqnarray}\label{effalgpuremotdecomp} \sAf^i(b)
\;\; (\; =\sZ^b_{b-i}\boxtimes h^{b})&=&\sum_{p_j}\sZ^b_{b-i}\boxtimes p_j \cdot h^b
\end{eqnarray}
%$$\sM\cong\oplus (\sM_{\sU}\otimes \sU).$$ 
where the sum runs through a set of projectors $p_j$ such that $\sum_j p_j=$Id, the identity projector.
%run through a set of 
%$h^{b}=\sum \sU$ is a decomposition of $h^b$ into irreducible representations $\sU$. 

Note that every irreducible subrepresentation $\sU$ of $h^b$ is of the form $p \cdot h^b$ for some projector $p$.

Note that if $p_1$ and $p_2$ are two projectors, $p_1\neq p_2$, but $p_1\cdot h^b\cong p_2 \cdot h^b$, we still have $\sZ^b_{b-i}\boxtimes p_1\cdot h^b\neq \sZ^b_{b-i}\boxtimes p_2 \cdot h^b$ as the cycles will depend on the choice of projector.

We now remark that the product structure preserves the decomposition of our algebra into pieces labeled by projectors.

Let $q\in \Q[\Sigma_{b+b'}]$ be an projector.  
It follows that we can define the \linebreak $q\cdot h^{\otimes{b+b'}}$-graded piece of $\sHf$
%:
%\begin{defn}
%$\chi({q\cdot h^{\otimes{b+b'}}}):=\Hom_{GL(h^1(E))}\big{(}\>\>
%%%\Sym^{n}\sH(-n-m)
%q\cdot h^{\otimes{b+b'}}\>\>,\>\>\chi\>\>\big{)}
%%%\otimes_{\Q} \Sym^{n}\sH(-n-m)
%$\end{defn}
componentwise in terms of direct summands of $(\sAf^{\otimes i})$.
%As discussed in the aside, we have 
We have
$$q \cdot (h(E^b)\otimes h(E^{b'}))=q \cdot ((h^1(E))^{\otimes b}\otimes(h^1(E))^{\otimes b'})=q \cdot ((h^1(E))^{\otimes b+b'})=q \cdot (h(E^{b+b'}))$$
which induces the equality 
\begin{align}\nonumber 
(\sZ^b_{b-i}\otimes\sZ^{b'}_{b'-i'})\otimes_{\Q[\Sigma_{b}\times \Sigma_{b'}]}q\cdot (h^{b}\otimes h^{b'}) = (\sZ^b_{b-i}\otimes\sZ^{b'}_{b'-i'})\otimes_{\Q[\Sigma_{b}\times \Sigma_{b'}]}q\cdot (h^{b+b'})
\end{align}

Notice that $q\in \Q[\Sigma_{b+b'}]$ does not act on $(\sZ^b_{b-i}\otimes\sZ^{b'}_{b'-i'})$ unless \linebreak
$q\in \Q[\Sigma_{b}\times \Sigma_{b'}]$.

When we combine the above equality with the algebra structure on $\sAf^{\bullet}$ 
%\ref{squishies} 
the product map restricts as follows:
%A projection is given by a choice of idempotent $\rho\in \Q[\Sigma_n]$
%since we are always taking $Alt$ wrt $(\Z/2\Z)^n$.

Let $\sV=p_1\cdot h^b$ and $\sW=p_2\cdot h^{b'}$ denote pure motives, and suppose $\sU\subset\sV\otimes\sW$ is an irreducible summand.  Given 
$X\boxtimes \sV \subset \sZ^b_{b-i}\boxtimes h^b$ and 
$Y\boxtimes \sW \subset \sZ^{b'}_{b'-i'}\boxtimes h^{b'}$, we have 
\begin{eqnarray*}
(X\boxtimes \sV)\otimes (Y\boxtimes \sW)&\mapsto& (X\cdot Y)\boxtimes (\sV\otimes\sW);
\end{eqnarray*}
%equivalently 
%\diagram(X\cdot p_1\boxtimes p_1\cdot h^b)\otimes (Y\cdot p_2\boxtimes p_2\cdot h^{b'})\hspace{.4in} \ddouble \\
%%\text{\rotatebox{270}{$=$}}\hspace{1 in}
%%\parallel \hspace{1.65in} & \\
%(X \otimes Y)\cdot (p_1\times p_2)\otimes_{\Q[\Sigma_b\times\Sigma_{b'}]} (p_1\times p_2)\cdot h^{b+b'} \dto|<\stop \\
%(X\cdot Y)\cdot (p_1\times p_2)\boxtimes (p_1\times p_2)\cdot h^{b+b'}
%\enddiagram
under the product map ($\star$).

%Notice that we have an induced product map on $GL_2$-representations:
%\begin{equation}
%\begin{array}{ccc}
%%\diagram
%(\sU\otimes\sW^{\vee})\otimes (\sU\otimes\sV^{\vee})&\subset&
% \sU \otimes (\sW\otimes \sV)\otimes (\sW^{\vee}\otimes\sV^{\vee}) \\
%%\\
%&\searrow&\cup\\
%&&\sU
%\end{array}
%%\enddiagram
%\end{equation}
%%(\sU \otimes \sU \otimes (\sW^{\vee}\otimes\sV^{\vee})
%%&&\sU \otimes (\sW\otimes \sV)\otimes (\sW^{\vee}\otimes\sV^{\vee})\\ 
%%\cap&&\cup \\

If $\sU=q\cdot (\sV\otimes\sW)$, the algebra map composed with the projection defined by $q$ induces a product 
%\ref{pthgraded}

\begin{eqnarray*}
(X\boxtimes \sW)\otimes (Y\boxtimes \sV)&\mapsto& (X\cdot Y)\boxtimes \sU;
\end{eqnarray*}

%\begin{eqnarray*}
%(X\boxtimes (\sU\otimes\sW^{\vee}))\otimes (Y\boxtimes (\sU\otimes\sV^{\vee}))&\mapsto& (X\cdot Y)\boxtimes \sU;
%\end{eqnarray*}

%equivalently 
%\diagram
%((X\cdot p_1\boxtimes p_1\cdot h^b)\otimes (Y\cdot p_2\boxtimes p_2\cdot h^{b'}))\boxtimes{\sU}\hspace{.2in} \ddouble \\
%%\text{\rotatebox{270}{$=$}}\hspace{1 in}
%%\parallel \hspace{1.85in}&& \\
%(X \otimes Y)\cdot (p_1\times p_2)\otimes_{\Q[\Sigma_b\times\Sigma_{b'}]} q\cdot(p_1\times p_2)\cdot h^{b+b'} \dto|<\stop \\
%(X\cdot Y)\cdot (p_1\times p_2)\cdot q\boxtimes q\cdot(p_1\times p_2)\cdot h^{b+b'}
%\enddiagram

Let 
$$((\sZ^b_{b-i}\boxtimes \sV)\otimes 
(\sZ^{b'}_{b'-i'}\boxtimes \sW))_{\sU}=(\sZ^b_{b-i}\otimes\sZ^{b'}_{b'-i'})(p_1\times p_2)\otimes_{\Q[\Sigma_{b}\times \Sigma_{b'}]}q\cdot (p_1\times p_2)\cdot (h^{b+b'})$$
Then the appropriate piece of the bar complex looks as follows:

\begin{equation}\label{pthgraded}
\begin{array}{cccc}
&\ddots\uparrow\partial && \uparrow \partial\\
\stackrel{\delta}{\to} &((\sZ^b_{b-i}\boxtimes \sV)\otimes 
(\sZ^{b'}_{b'-i'}\boxtimes \sW))_{\sU}& 
\stackrel{\delta}{\to} & \sZ^{b+b'}_{b+b'-i-i'}\boxtimes \sU \\
&\uparrow\partial&\ddots &\uparrow\partial \\
&&\stackrel{\delta}{\to} & \sZ^{b+b'}_{b+b'-i-i'+1}\boxtimes \sU \\
\end{array}
\end{equation} 

Here the diagonal dots indicate the $H^0$ diagonal of $B(\sAf^{\bullet})$ (although if $\sHf$ turns out not to be a $K(\pi,1)$ then for the purposes of this discussion the dots could represent any $H^i$ diagonal).

In particular the above discussion makes sense if 
$q\cdot (h^{b+b'})$ is equal to an irreducible pure motive $\Sym^{n}h^1(E)(-m)$.

Furthermore, when we pass to the Lie co-algebra $\sM$, it therefore makes sense to talk about 
% for future reference 
the $\Sym^{n}h^1(E)(-m)$-graded pieces of $\sM$.  

Let $I\subset H^0(B(\sAf))$ denote  the augmentation ideal.  
%Define 
%for future reference
\begin{defn}\label{boldB1}
%\begin{eqnarray*}
%I_{\Sym^{n}h^1(E)(-m)}&:=&I(\Sym^{n}h^1(E)(-m))\\
%\Hom_{GL(h^1(E))}\big{(}\>\>\Sym^{n}h^1(E)(-m)\>\>,\>\>I\>\>\big{)}
%\otimes_{\Q} \Sym^{n}h^1(E)(-m)
%$$
%\end{defn}
%Also define
%\begin{defn}
%$$
$\sM_{\Sym^{n}h^1(E)(-m)}:= 
%\Hom_{GL(h^1(E))}\big{(}\>\>\Sym^{n}h^1(E)(-m)\>\>,\>\>I/I^2\>\>\big{)}
%\otimes_{\Q} \Sym^{n}h^1(E)(-m)
I({\Sym^{n}h^1(E)(-m)})/I^2({\Sym^{n}h^1(E)(-m)})$
%(I_{\Sym^{n}h^1(E)(-m)}\cdot I_{\Sym^{n}h^1(E)(-m)}) 
%$$
%\end{eqnarray*}
%where $I\subset H^0(B(\sAf))$ denotes the augmentation ideal.  
\end{defn}

%where $I^2_{\Sym^{n}h^1(E)(-m)}:=\Hom_{GL(h^1(E))}\big{(}\>\>\Sym^{n}h^1(E)(-m)\>\>,\>\>I^2\>\>\big{)}$.

%\subset I^2$ consists of all products in $I_{\Sym^{n}h^1(E)(-m)}$.
%restriction of the square of the augmentation ideal to $I_{\Sym^{n}h^1(E)(-m)}$.
%Use statement from section 2.  consists of all pairs ...

\begin{remark}
In the last chapter of this paper (where we compute explicit motivic elements) there will be no need to pass to the full category of mixed elliptic motives.  In that section of the paper we will therefore work in the category $\fM(E)^{\eff}$ of effective mixed elliptic motives. 
\end{remark}

\subsection{Mixed elliptic motives: \\The Definition of the Motivic DGA $\sA_E$ \\and of the Motivic Hopf Algebra $\sH_E$)}\label{realdef}

Consider the following algebra (all conventions as in the previous section):

%\begin{defn}
\begin{eqnarray*}
\sA_{aug}&:=&\sum_{i,a,b}\sA^i_{aug}(a,b)\quad\quad\text{ where}\\
%\sA_{\fM(E)^{aug}}^i(a,b)
\sA_{aug}^i(a,b)
&:=&\text{\rm Alt}_{(\Z/2\Z)^{b}}(\sZ^{b-a}({E}^{b},b-2a-i))\otimes_{\Q[\Sigma_b]}(h^1(E)^{\otimes b})(a)\\
%&=:&\sZ^b_{b-i}\otimes_{\Q[\Sigma_b]}h^{b} 
%\big{(}
&=:&\sZ^{b,a}_{b-2a-i}\boxtimes h^{b}_a
\end{eqnarray*}
%\end{defn}

Given a positive integer b, we say that a pure elliptic motive $\sV$ is {\it effective of weight b} if it is a direct summand of $h^1(E)^{\otimes b}$.  Notice that any pure elliptic motive $\sW$  can be written  $\sW=\sV(a)$ as some positive twist $\Q(a)$ of an effective motive  $\sV$ of weight $b$ for some $b$.
Clearly $(\sW(-c))(c)=(\sV(-c))(a+c)$, where $(\sV(-c))$ is effective of weight $b+2c$ and $c>0$, is another way to write the same motive $\sW$, 
and just as clearly these are the only ways to decompose $\sW$ into a positive twist of an effective motive.

The above paragraph becomes substantive when we adjoin cycle groups to each pure motive.  
Namely, $\sZ^{n,0}_{n-i}\boxtimes \sV$ may be different from 
\linebreak $\sZ^{n+2a,a}_{n-i}\boxtimes (\sV(-a))(a)$, which a priori suggests a well-definedness problem.  Thus we modify the construction of $\sA_{aug}$ by taking a direct limit to remove this ambiguity.

As in the last section, the algebra $\sA_{aug}$ 
%\;\;(\;= \sum_i\sum_b\sA_{\fM(E)^{aug}}^i(a,b))$ 
is graded by irreducible representations of $GL_2$.
Indeed, 
%since every irreducible representation $\sU$ of $h^b$ is of the form $p \cdot h^b$ for some projector $p$,
we have a direct sum decomposition 
\begin{eqnarray}\label{algpuremotdecomp} 
%\sA_{\fM(E)^{aug}}^i(a,b)(=
\sZ^{b,a}_{b-2a-i}\boxtimes h^{b}_a&=&\sum_{p_j}\sZ^{b,a}_{b-2a-i}\boxtimes (p_j(h^b))(a)
\end{eqnarray}
%$$\sM\cong\oplus (\sM_{\sU}\otimes \sU).$$ 
where the sum runs through a set of projectors $p_j$ such that $\sum_j p_j=$Id, the identity projector.

Note that $\sum_j p_j(h^1(E)^{\otimes b})(a)=h^1(E)^{\otimes b}(a)$ only if $a=0$.  

For a positive integer $c$, 
%and a cycle $\sC\in \sZ^{b,a}_{b-2a-i}$, 
we have an inclusion of cycle groups 
\begin{eqnarray}\label{inclusionproduct}
\sZ^{b,a}_{b-2a-i}\hookrightarrow \sZ^{b+2c,a+c}_{b-2a-i}
\end{eqnarray}
which is induced via the (external) product map on cycles 
\begin{eqnarray*}
(\sZ^{2,1}_{0})^{\otimes c} \otimes \sZ^{b,a}_{b-2a-i}&\rightarrow& \sZ^{b+2c,a+c}_{b-2a-i}\\
(\Delta-\Delta^{-})^{\otimes c}\otimes \sC&\rightarrow & \prod_{i=1}^c(\Delta_{2i-1}-\Delta^{-}_{2i-1})\cdot\sC
\end{eqnarray*}
where $\Delta$ and $\Delta^{-}$ denote respectively the diagonal and the antidiagonal of $E\times E$, where $\Delta_i$ denotes the i-th diagonal, $\Delta^{-}_j$ denotes the j-th antidiagonal\footnote{If $P\in E^{2a}$ is a closed point of form $(x_1,x_2,\ldots,x_{2a})$, then $\Delta_i\subset E^{2a}$ is the set of all $P$ for which $x_i-x_{i+1}=0$, and $\Delta^{-}_j\subset E^{2a}$ is the set of all $P$ for which $x_j+x_{j+1}=0$.}, 
%$\sC\in \sZ^{b,a}_{b-2a-i}$, 
and product is
cycle-theoretic intersection.

%$(\sW(-c))(c)=(\sV(-c))(a+c)$, where $(\sV(-c))$ is effective of weight $b+2c$ and $c>0$
%CH^{2m+2a+n-a}(E^{2m+2a+n}, 2m+2a+n-2a-1)\boxtimes \Sym^nh^1(E)(-m-a)(a) 
Let $\sV=q\cdot h^1(E)^{\otimes b}$ be an effective irreducible representation of $GL_2$ (with associated projector $q$), let $\sB_{a,c}=\sZ^{b+2c,a+c}_{b-2a-i}\boxtimes h^{b+2c}_{a+c}$ and $\pi_{q,c}\sB_{a,c}=\sZ^{b+2c,a+c}_{b-2a-i}\boxtimes \sV(a)$ be the projection.  Note that we have inclusion maps $\pi_{q,c}\sB_{a,c}\hookrightarrow \sB_{a,c}$ and projection maps $\sB_{a,c}\twoheadrightarrow \pi_{q,c}\sB_{a,c}$ for all $c$.  Furthermore, we have inclusion maps $\sB_{a,n}\hookrightarrow\sB_{a, n+c}$ induced from the inclusion of cycle groups (\ref{inclusionproduct}) above, 
%, let $\pi_q:\sB_{a,c} \rightarrow \sB_{a,c,q}$ denote the projection ($\pi_q(\sB_{a,c})=\sB_{a,c,q}$), and let $i_q:q\sB_{a,c\hookrightarrow \sB_{a,c}$ denote the inclusion.  

Diagrammatically, we have  
%\begin{figure}[htbp]
%{
$$
\xymatrix{
\pi_{q,0}\sB_{a,0} \ar@{^{(}->}[r] & \sB_{a,1} \ar @{^{(}->}[r] \ar @<1ex> @{^{>>}} [d] & \sB_{a,2} \ar @{^{(}->}[r] \ar @<1ex> @{^{>>}} [d]&\cdots\\
& \pi_{q,1}\sB_{a,1} \ar @<1ex> @{^{(}->}[u]& \pi_{q,2}\sB_{a,2} \ar @<1ex> @{^{(}->}[u]&}
$$
%}
%\caption{Diagram 1}
%\label{diagram1}
%\end{figure}

which induces
%{$$
%\xymatrix{ 
%\pi_{q,0}\sB_{a,0} \ar[r]^{\sim} 
%\ar@/_2pc/[ddrr]_{\sim} & \pi_{q,1}\sB_{a,1} \ar[d]^{\sim} \ar[ddr]^{\sim}&& \\
%& \pi_{q,2}\sB_{a,2}  \ar[dr]^{\sim} && \\
%&& \pi_{q,3}\sB_{a,3} \ar@{-->}[dr]&\\
%&&&\\}
%$$}

%where all the morphims are quasi-isomorphisms.

$$
\xymatrix{ 
\pi_{q,0}\sB_{a,0} \ar[r] 
\ar@/_2pc/[ddrr] & \pi_{q,1}\sB_{a,1} \ar[d] \ar[ddr]&& \\
& \pi_{q,2}\sB_{a,2}  \ar[dr] && \\
&& \pi_{q,3}\sB_{a,3} \ar@{-->}[dr]&\\
&&&\\}
$$

In the following section (section \ref{sec-prop}) we will prove (Prop \ref{compprop}) that, with notation as above,

%\begin{lem}\label{complem} 
%\begin{multline}
$CH^{2m+n}(E^{2m+n}, 2m+n-1)\boxtimes \Sym^nh^1(E)(-m) \cong $\\
%to $CH^{2m+n}(E^{2m+n}, 2m+n-1)$ is 
$\big(\prod_{i=1}^m(\Delta_{2i-1}-\Delta^{-}_{2i-1})\big)\cdot
 CH^{n+m}(E^{n},2m+n-1)\cdot \varrho_{n,0}\boxtimes \Sym^nh^1(E)(-m),$
%\end{multline}

%where $\Delta_i$ denotes the i-th diagonal, $\Delta^{-}_j$ denotes the j-th antidiagonal \footnote{If $P\in E^{2a}$ is a closed point of form $(x_1,x_2,\ldots,x_{2a})$, then $\Delta_i\subset E^{2a}$ is the set of all $P$ for which $x_i-x_{i+1}=0$, and $\Delta^{-}_j\subset E^{2a}$ is the set of all $P$ for which $x_j+x_{j+1}=0$}, and product is
%cycle-theoretic intersection.
%\end{lem}

It follows that   
%In particular notice that 
for large enough $c$, $\pi_{q,c}\sB_{a,c}\stackrel{\sim}{\rightarrow}\pi_{q,d}\sB_{a,d}$ is a quasi-isomorphism for all $c<d$.

\begin{defn}  We have
\begin{eqnarray*}
\sA_E^{\bullet}&:=&\sum_i\lim_{\rightarrow}\Big{(}\sum_b\sum_{0\leq a<b}\sum_{p_j}\sA_E^i(a,p_j)\Big{)},
\end{eqnarray*} where
\begin{eqnarray*}
%\noindent
\sA_E^i(a,q)
&:=&
\lim_{\stackrel{\rightarrow}{c}}\pi_{q,c}\sB_{a,c}.
%\\
%&=:&\sZ^b_{b-i}\otimes_{\Q[\Sigma_b]}h^{b} 
%\big{(}
%&=:&\sZ^{b,a}_{b-2a-i}\boxtimes h^{b}_a
\end{eqnarray*}
Here the limits are filtered colimits taken over all diagrams of the form listed above, and the sum over $p_j$ runs through a set of projectors $p_j$ such that $\sum_j p_j=$Id, the identity projector.

Equivalently, 
\begin{eqnarray*}
\sA_E^{\bullet}&:=&\sum_i\sum_q\sA_E^i(a,q)
\end{eqnarray*}
where the sum over $q$ runs over all irreducible objects in the category of linear representations of $Gl_2$. 
\end{defn}

Note that since the limits being taken are filtered they will be exact, and so commute with arbitrary sums and cohomology.

%\begin{defn}
%\begin{eqnarray*}
%\noindent\sA_{\widehat\fM(E)}^i(a,b,q)
%:=\quad\quad\quad\hfill\\
%\lim_{\stackrel{\rightarrow}{c}}[\text{\rm Alt}_{(\Z/2\Z)^{b+2c}}(\sZ^{b+2c-a-c}({E}^{b+2c},b-2a-i))\otimes_{\Q[\Sigma_b]}q(h^1(E)^{\otimes b+2c})(a+c)]
%%%%\\
%%%%&=:&\sZ^b_{b-i}\otimes_{\Q[\Sigma_b]}h^{b} 
%%%%\big{(}
%%%%&=:&\sZ^{b,a}_{b-2a-i}\boxtimes h^{b}_a
%\end{eqnarray*}
%\end{defn}

%\begin{defn}
%\begin{eqnarray*}
%\sA_{\widehat\fM(E)}^{\bullet}&:=&\sum_i\lim_{\rightarrow}\Big{(}\sum_b\sum_{0\leq a<b}\sum_{p_j}\sA_{\widehat\fM(E)}^i(a,p_j)\Big{)}
%\end{eqnarray*}
%here again the limit is taken over all filtered colimits as above, and the sum over $p_j$ runs through a set of projectors $p_j$ such that $\sum_j p_j=$Id, the identity projector.
%\end{defn}

%Equivalently:

%\begin{defn}
%\begin{eqnarray*}
%\sA_{\widehat\fM(E)}^{\bullet}&:=&\sum_i\sum_q\sA_{\widehat\fM(E)}^i(a,q)
%\end{eqnarray*}
%where the sum over $q$ runs over all irreducible objects in the category of linear representations of $Gl_2$. 
%\end{defn}

Also note that the product structure discussed in the last section extends in a natural way to a product structure on $\sA_E^{\bullet}$.

%$\widehat\fM(E)$

\begin{defn}
%Define 
$\sH_E=H^0(B(\sA_E^{\bullet}))$.  The associated Lie coalgebra 
 $\sM$ is $ I/ I^2$ where $I$ denotes the augmentation ideal of $\sH_E$.

The category $\fM(E)$ of mixed elliptic motives is defined to be the category of 
%\linebreak 
comodules over $\sH_E$, or (equivalently) 
the category of 
co-representations of the Lie co-algebra $\sM$.
\end{defn}
%better notation?is $L_{\chi}$

\begin{remark}  It follows from Proposition \eqref{compprop} from the next chapter that 
$$CH^n(E^{2n},0)\boxtimes \Q \cong \Q$$ and hence $$\lim_{\rightarrow}CH^n(E^{2n},0)\boxtimes \Q\cong \Q.$$  In other words, the augmentation map still maps to the coefficient ring $\Q$.
\end{remark}

%\begin{defn}\end{defn}

%*********************
Theorem 2.4 (p.\ 81) and Theorem 1.2 (p.\ 77) from \cite{KM}  imply the following proposition:
\begin{prop} Assume the Beilinson-Soul\'e conjecture for $E$.  Then the category $\fM(E)$ is equivalent to the heart $\fH_{\sA_E^{\bullet}}$ of the derived category 
%$\fH_{\sA_{\widehat\fM(E)}^{\bullet}}$, where  is the heart of the derived category 
$\fD_{\sA_E^{\bullet}}$ with respect to the $t$-structure defined in Definition 4.1 (p.\ 85) from \cite{KM}.
\end{prop}
It would be very interesting to relate $\fD_{\sA_E^{\bullet}}$ to the categories defined by Voevodsky, Levine, and Hanamura.
%\end{remark}

%*****************Section 2***************

\section{A General Cycle Group Computation}\label{chap:CC}

%\subsection{Introduction}\label{sec:intro} 
In this section we will compute the cohomology of our DGA
%s $\sA_E$ and 
$\sA_E$.  We then relate these cohomology groups to the Ext-groups of our category $\fM(E)$ assuming that $\sA_E$ is cohomologically supported in degree one.
%a $K(\pi,1)$.
%\footnote{
%These computations should generalize to the case of a generic curve of arbitrary genus}.

\subsection{Motivic Complexes}\label{sec:MC}

In order to compute the Ext-groups of our category, we consider the
associated category of representations of the associated Lie co-algebra. 
A basic result for Lie coalgebras (see for example \cite{We} pp.\ 224-5) is that the 
Ext-groups of the
 category of co-representations of a
graded Lie coalgebra, such as $\sM$, can be computed by looking at the complex

$$\sM\rightarrow \bigwedge^2\sM\rightarrow \bigwedge^3\sM \cdots$$
where the maps are given by the differential on the
Lie-coalgebra.

%Must define $\bold B$ here if not done in Section 1.

As discussed in the previous section (see \eqref{algpuremotdecomp}), our Lie coalgebra is labelled by irreducible linear representations $\sU$ of $GL_2$.  Thus we have a direct sum decomposition 
$$\sM\cong\oplus (\sM_{\sU}\otimes \sU).$$  Hence $\bigwedge^2\sM$ breaks up into a direct sum decomposition
\begin{multline}
\bigwedge^2\sM\cong \bigwedge^2(\oplus (\sM_{\sU}\otimes \sU)) 
\cong \\
\sum (\sM_{\sV}\bigwedge \sM_{\sW})\otimes (\Sym(\sV\otimes\sW))
\bigoplus (\Sym(\sM_{\sV}\otimes\sM_{\sW}))\otimes (\sV\wedge\sW)
\end{multline}
where 
$$\sM_{\sV}\bigwedge \sM_{\sW}=
\Big{\{}\mycases
{\bigwedge^2 \sM_{\sV}\>\>\>\>\>\>\>\>\>\text{ if }\sV=\sW}
{\sM_{\sV}\otimes \sM_{\sW}\>\text{ if }\sV\neq \sW},$$
$$\sV\wedge \sW\>\>=
\Big{\{}\mycases
{\bigwedge^2 \sV\>\>\>\>\>\text{ if }\sV=\sW}
{\sV\otimes \sW\>\text{ if }\sV\neq \sW},$$
$$\Sym(\sM_{\sV}\otimes\sM_{\sW})=
\Big{\{}\mycases
{\Sym(\sM_{\sV})\>\>\>\>\>\text{ if }\sV=\sW}
{\sM_{\sV}\otimes \sM_{\sW}\>\>\text{ if }\sV\neq \sW},$$
$$\Sym(\sV\otimes\sW)=
\Big{\{}\mycases
{\Sym(\sV)\>\>\text{ if }\sV=\sW}
{\sV\otimes \sW\>\>\>\text{ if }\sV\neq \sW},$$
%where Sym$A\otimes B)$ denotes the symmetric product of $A$ and $B$,  
and where the sum runs over all pairs $(\sV, \sW)$ of irreducible representations. 
%$\Sym^2$ denotes the symmetric square,

Thus
%, for our bigraded Lie coalgebra, 
the Ext-groups of the category of co-representations of $\sM_E$ labelled by an irreducible representation $\Sym^nh^1(E)(a-m)$
are given by the 
%hyper
cohomology of the 
%bi
complex 

\begin{multline}\label{liealgdiffcomp}
\sM_{\Sym^nh^1(E)(-m)}\otimes \Sym^nh^1(E)(-m)\stackrel{\bar{\text{d}}}{\longrightarrow
} \\
\sum_{\Sym(\sV\otimes\sW)} (\sM_{\sV}\bigwedge \sM_{\sW})\otimes (\Sym(\sV\otimes\sW))
\bigoplus \sum_{\sV\wedge\sW}(\Sym(\sM_{\sV}\otimes\sM_{\sW}))\otimes (\sV\wedge\sW)
\\
%\big{(}\sM_{\sV}\bigwedge\sM_{\sW}\big{)}\otimes \big{(}\sV\wedge \sW \big{)}
\longrightarrow \cdots
\end{multline}
where $\sV$ and $\sW$ are irreducible representations 
%in the category of pure elliptic motives
, the first sum runs through all pairs $(\sV,\sW)$ such that 
%$$\big{\{}\Sym(\sV\otimes\sW)|\sV \subset h(E^{k}), \sW \subset h(E^{n+2m-k})\text{ for some }k, 1\leq k\leq n+2m-1\big{\}}$$ such that  
$$\Sym^nh^1(E)(a-m)\subset \Sym(\sV\otimes\sW)$$
is a summand in a direct sum decomposition of $\Sym(\sV\otimes\sW)$ into irreducible representations,
and the second sum runs through all pairs $(\sV,\sW)$ such that
%$$\big{\{}\sV\wedge \sW|\sV \subset h(E^{k}), \sW \subset h(E^{n+2m-k})\text{ for some }k, 1\leq k\leq n+2m-1\big{\}}$$ such that 
$$\Sym^nh^1(E)(a-m)\subset \sV \wedge \sW$$
is a summand in a direct sum decomposition of $\sV \wedge \sW$ into irreducible representations.\footnote{I thank A. Levin for pointing out an error in the definition of complex \eqref{liealgdiffcomp} in a preliminary version of this paper.}

It is easy to show that 
%\begin{prop}\label{pointextensions}
\begin{eqnarray}
\text{Ext}^1_{\sM(E)}(h^1(E),\Q)\cong E(k)\otimes \Q.
\end{eqnarray}
%\end{prop}

The goal of the next two sections is to prove the following proposition and remark on its consequences. 
\begin{prop}\label{compprop}  Let $a,n,m\in \Z, n,m\geq 0$.  Assume that $\sA_E$ is cohomologically connected with respect to $i$.
%quasi-isomorphic to its 1-minimal model.
%$\sA_E$ is cohomologically connected with respect to $i$.  
Then the appropriate component of the kernel of the map $\bar{\text{d}}$ in \eqref{liealgdiffcomp} is isomorphic to
$$(CH^{m+n-a}(E^n,{2m+n-2a-1})\otimes\Q)_{sgn}\boxtimes \Sym^nh^1(E)(a-m)$$ where sgn denotes the sign character eigenspace for
the natural action of the symmetric group on $E^n$.
\end{prop}

\subsection{The Cycle Group Computation}\label{sec-prop}
%The idea below is to look inductively at the sequence of inclusions 
%$$\varrho_{n,m}\subset \varrho_{n,m-1}\otimes \varrho_{0,1}\subset\cdots
%\subset \varrho_{n,0}\otimes \varrho_{0,1}^{\otimes m}$$

In general, the decomposition of the tensor product of two projectors into projectors is complicated.  However, if  $\Bbb{S}_{\lambda}V=\varphi_1V^{\otimes |\lambda|}$ and $\Bbb{S}_{\kappa}V=\varphi_2V^{\otimes |\kappa|}$ denote two irreducible $GL(V)$-representations (with corresponding projectors $\varphi_1$ and $\varphi_2$), then
$$ \Bbb{S}_{\lambda}V\otimes\Bbb{S}_{\kappa}V
= \varphi_1V^{\otimes |\lambda|}\otimes \varphi_2V^{\otimes |\kappa|}
= (\varphi_1\otimes \varphi_2)V^{\otimes |\lambda|+|\kappa|}$$
and there is a standard decomposition
\begin{align}\label{YDtensordecomp}(\varphi_1\otimes\varphi_2)\cong \sum \tilde{\xi}\end{align}
(see for example \cite{FH} Lecture 6 and exercises therein) where the sum runs over a certain set of projectors\footnote{given up to Young diagrams i.e.  $\Bbb{S}_{\lambda}V\otimes \Bbb{S}_{\kappa}V\cong \oplus_{\tilde{\xi}}N_{\lambda\kappa\tilde{\xi}}\Bbb{S}_{\tilde{\xi}}V$, where the numbers $N_{\lambda\kappa\tilde{\xi}}$ are given by the {\it Littlewood-Richardson rule} } to irreducible representations associated to Young diagrams of size $|\lambda|+|\kappa|$.  
Let ${\xi}=\tilde{\xi}\otimes \text{Alt}$ where $\tilde{\xi}\in \Q[\Sigma_{|\lambda|+|\kappa|}]$ is a projector.  Let $\sZ^{a,b}=\sum_i\sZ^{b,a}_{b-2a-i}$.  We get a map of $GL_2$-representations 
$$(\sZ^{a,b}\boxtimes (\rho_1\otimes\rho_2)\cdot h^{b}_a)\stackrel{\sim}{\rightarrow}
(\sZ^{a,b}\boxtimes \sum {\xi} \cdot h^{b}_a)
=\sum (\sZ^{a,b} \cdot {\xi}\boxtimes {\xi} \cdot h^{b}_a)$$

Consequentially we get a decomposition 
\begin{multline}\label{nonalgdecomp}
\sA^i_{E}(a,b)\cong\sum_{\Sym^{n}h^1(E)(a-m)}
\rho^{t}_{n,m}(\sZ^{a,b}(
{E}^{b},b-2a-i))
\otimes \Sym^{n}h^1(E)(a-m) \\
(:=
\sum_{\Sym^{n}h^1(E)(a-m)}
\sZ^{b}(
{E}^{b},b-2a-i)\boxtimes \Sym^{n}h^1(E)(a-m) ) 
\end{multline}
where we have chosen a decomposition of $h^{b}(E^{b})$ into irreducible representations via projectors $\varrho_{n,m}$ ($n+2m=b$)
%(compatible with choices for $h^{n}(E^{n}),n\leq m$)\footnote{all diagrams of the form 
%$$\begin{diagram}\end{diagram}$$,
%^{a+c}\sH(-a-c-b-d)\rto
and $\sum_{\Sym^{n}h^1(E)(a-m)}$ denotes the sum over all of the irreducible representations of that decomposition.  See Section \ref{projectornotation} for an explicit choice of projectors $\varrho_{n,m}$.

Since $\bigwedge^nh^1(E)=0$ for $n\geq 3$, it follows that only one term in the decomposition of $({\varrho_{0,1}\otimes \varrho_{n,m-1}})\cdot h^{n+2m}$ defines a nonzero projection.  Hence we get an isomorphism of linear $GL_2$-representations \linebreak $(\varrho_{0,1}\otimes \varrho_{n,m-1})\cdot h(E^2\times E^{n+2m-2})\cong\varrho_{n,m}\cdot h(E^{n+2m})$.

When we pass to the cycle side, however, notice that we have an {\it equality} \linebreak $\varrho^t_{n,m}\cdot (\varrho^t_{0,1}\otimes \varrho^t_{n,m-1})(\sZ^{n+2m})=\varrho^t_{n,m}(\sZ^{n+2m})$.  
It follows by induction that we have 
%(non-canonical) 
equalities
\begin{eqnarray*}
\varrho^t_{n,m}(\sZ^{n+2m}) &=&\varrho^t_{n,m}\cdot(\varrho^t_{0,1}\otimes \varrho^t_{n,m-1})(\sZ^{n+2m})=\dots \\
&=&\varrho^t_{n,m}\cdot((\varrho^t_{0,1})^{\otimes m}\otimes \varrho^t_{n,0})(\sZ^{n+2m}). 
\end{eqnarray*}
%$$\varrho_{n,m}\cdot h(E^{n+2m})=\varrho_{n,m}\cdot(\varrho_{0,1}\otimes \varrho_{n,m-1})\cdot h(E^{n+2m})=\dots=\varrho_{n,m}\cdot(\varrho_{0,1}^{\otimes m}\otimes \varrho_{n,0})\cdot h(E^{n+2m})\>\>.$$
%(though this isomorphism is generally non-canonical)

We will show that 
Proposition \ref{compprop} is a consequence of the following lemma:

\begin{lem}\label{complem} 
%\begin{multline}
$CH^{2m+n}(E^{2m+n}, 2m+n-1)\boxtimes \Sym^nh^1(E)(-m) \cong $\\
%to $CH^{2m+n}(E^{2m+n}, 2m+n-1)$ is 
$\big(\prod_{i=1}^m(\Delta_{2i-1}-\Delta^{-}_{2i-1})\big)\cdot
 CH^{n+m}(E^{n},2m+n-1)\cdot \varrho_{n,0}\boxtimes \Sym^nh^1(E)(-m),$
%\end{multline}

where $\Delta_i$ denotes the i-th diagonal, $\Delta^{-}_j$ denotes the j-th antidiagonal 
%\footnote{See \footnote[6]}
%\footnote{If $P\in E^{2a}$ is a closed point of form $(x_1,x_2,\ldots,x_{2a})$, then $\Delta_i\subset E^{2a}$ is the set of all $P$ for which $x_i-x_{i+1}=0$, and $\Delta^{-}_j\subset E^{2a}$ is the set of all $P$ for which $x_j+x_{j+1}=0$}
, and product is
cycle-theoretic intersection.  (Notation is as in Section \ref{realdef}, footnote 6.)
\end{lem}
\begin{proof}
%Factor the projection $\varrho_{n,m}$
We have the following commutative diagram
:

% original attempt at diagram (xypic):
%$$\objectmargin{0.5pc}
%\diagram CH^{2m+n}(E^{2m+n}, 2m+n-1)
% \dto!<5\cR,\cD>^<<<<<<<{\varrho^t_{0,1}\otimes\varrho^t_{n,m-1}} 
%\ddtol_<<<<<<<<<<<{\varrho^t_{n,m}}\\
%\save \go+<2cm,0cm>\Drop{CH^{2m+n}(E^2\times E^{2m+n-2}, 2m+n-1)
%\cdot(\varrho_{0,1}\otimes\varrho_{n,m-1})}
%\dto^<<<<<{\varrho^t_{n,m}}  \\
%\quad \quad \rightarrow CH^{2m+n}(E^{2m+n}, 2m+n-1)\cdot\varrho_{n,m}
%\restore
%\enddiagram$$

$$
\xymatrix{
CH^{2m+n}(E^{2m+n}, 2m+n-1)
\ar@/_13pc/[dd]_{\varrho^t_{n,m}}
\ar[d]^{\varrho^t_{0,1}\otimes\varrho^t_{n,m-1}}\\
CH^{2m+n}(E^2\times E^{2m+n-2}, 2m+n-1)\cdot(\varrho_{0,1}\otimes\varrho_{n,m-1}) \ar[d]^{\varrho^t_{n,m}}\\
CH^{2m+n}(E^{2m+n}, 2m+n-1)\cdot\varrho_{n,m}}
$$

%% The other good choice for the diagram.  The long arrow gets broken up so not as good as above.
%$$
%\xymatrix{
%\text{---------------} CH^{2m+n}(E^{2m+n}, 2m+n-1)\quad\quad\quad\quad
%\ar@{-} `l[d] `[dd]_{\varrho^t_{n,m}} 
%\ar[d]^{\varrho^t_{0,1}\otimes\varrho^t_{n,m-1}}\\
%CH^{2m+n}(E^2\times E^{2m+n-2}, 2m+n-1)\cdot(\varrho_{0,1}\otimes\varrho_{n,m-1}) \ar[d]^{\varrho^t_{n,m}}\\
%%%%%%%%%%%\text{--------------}\ar{>}
%\text{------------}
%\longrightarrow CH^{2m+n}(E^{2m+n}, 2m+n-1)\cdot\varrho_{n,m}}\quad\quad\quad\quad\quad\quad\quad
%$$

%$$
%\xymatrix{
%CH
%\ar `l[d] `[dd]_{\varrho^t_{n,m}} 
%\ar[d]^{\varrho^t_{0,1}\otimes\varrho^t_{n,m-1}}\\
%HAHAHAHAHA \ar[d]^{\varrho^t_{n,m}}\\
%AD}
%$$

%\boxtimes  h(E^{2m+n})
%\boxtimes \Big((\varrho_{0,1}\otimes\varrho_{n,m-1})h(E^{2m+n})\Big)}
%\boxtimes \Sym^nh^1(E)(-m)

Notice that the $\varrho^t_{0,1}$-projection simply requires that the switch map \linebreak
$\sigma_2(x,y)=(y,x)$ is the
identity on the $E^2$ factor corresponding to $\phi^t_{0,1}$, 
and that the appropriate factor of 
$(\Z/2\Z)^2\subset (\Z/2\Z)^n$ acting on $E\times E$ projects via the sign character.  Note that the action of the dihedral group (and all its subgroups) on $E\times E$ is proper (all quotients are (possibly singular) quasi-projective varieties).  It follows (for example from \cite{EG}, Theorem 2, p.\ 14) that we can pass to the appropriate quotient variety.  Thus, we wish to compute 
$$(Id\otimes \varrho^t_{n,m-1})\cdot(CH^{2m+n}(P\times E^{2m+n-2}, 2m+n-1)^{\iota_P=+1})\boxtimes\Sym^nh^1(E)(-m),$$ where
$P:=\Sym^2(E)$ is a $\P^1$-bundle over $J(E)=E$, and $\iota_P$ is the automorphism induced by the
automorphism $\iota(x,y)=(-x,-y)$ on $E\times E$.  Applying the projective bundle theorem yields
\begin{multline}\label{p11}
CH^{2m+n}(P\times E^{2m+n-2}, 2m+n-1)^{\iota_P=+1}\cong \\
CH^{2m+n}(E\times E^{2m+n-2}, 2m+n-1)^{\iota_E=+1}\oplus \\
CH^{2m+n-1}(E\times E^{2m+n-2}, 2m+n-1)^{\iota_E=+1}\cdot c_1(\sO_P(1))
\end{multline}

(To save space I am omitting $\boxtimes\Sym^nh^1(E)(-m)$ throughout this sequence of reductions.)

 Since $Q:=E/\iota\cong\P^1$ is a $\P^1$-bundle over a point, we can again apply the projective bundle
theorem to each factor in the above decomposition;

\begin{multline}\label{p12}
CH^{2m+n-j}(Q\times E^{2m+n-2}, 2m+n-1)\cong \\
CH^{2m+n-j}(E^{2m+n-2}, 2m+n-1)\oplus
CH^{2m+n-j-1}(E^{2m+n-2}, 2m+n-1)\cdot c_1(\sO_Q(1)).
\end{multline}

Let's review.  The above computations show that we have the following decomposition:

%use that $\varrho^t_{0,1}\circ\varrho_{0,1}=8\varrho_{0,1}$?

\begin{multline}\label{projdecomp}
\varrho^t_{0,1}(CH^{2m+n}(E^{2m+n}, 2m+n-1)) \\ \cong 
\Big(CH^{2m+n}(E^{2m+n-2}, 2m+n-1) 
\oplus CH^{2m+n-1}(E^{2m+n-2}, 2m+n-1)^{\oplus 2}\\ 
\oplus CH^{2m+n-2}(E^{2m+n-2}, 2m+n-1)\Big).
\end{multline}

Here we think of $\varrho^t_{0,1}$ as the projection 
$\varrho^t_{0,1}\otimes Id$.

Let $\Delta, \Delta^{-}$ denote the classes of the diagonal
 and antidiagonal, respectively.
We have a distinguished section $s:E\rightarrow P$ of $P\rightarrow E$ given by the image (under the natural map $E\times E\stackrel{\psi}{\rightarrow}\Sym^2(E)$) of the diagonal embedding of $E$ into $E\times E$.  It follows (see, for example, \cite{H} Propositions 2.6 and 2.8 pp.\ 371-372) that $\sO_P(1)\cong\sL(s(E))$.  Since $\psi$ has degree 2, we have 
%Notice that 
%$$4c_1(\sO_P(1))=c_1(\sO_P(1)^{\otimes 4})=c_1(\sO_P(4))$$
$$c_1(\psi^*(\sO_P(1)))=\psi^*c_1(\sO_P(1))=\psi^*(s(E))=2\Delta.$$
%pulls back to $\Delta\subset E\times E$ under the sum map $E\times E \rightarrow E$.
Similarly, $c_1(\sO_Q(1))$ pulls back to $4\Delta^{-}\subset E\times E$ under the 
composite of the sum map and the projection 
%$E\rightarrow (Q\rightarrow \text{pt})$.
$E\rightarrow Q$ (both of which are manifestly degree 2 maps).
Since we are working with rational coefficients (though all we really need in this computation is for 2 to be invertible), we can choose 
the pullback of \eqref{projdecomp} to 
$CH^{2m+n}(E^{2m+n}, 2m+n-1)$ to be 

\begin{multline}\label{deltapsi}
CH^0(E^2)\cdot CH^{2m+n}(E^{2m+n-2}, 2m+n-1)\\
\oplus \Delta\cdot CH^{2m+n-1}(E^{2m+n-2}, 2m+n-1)\oplus\Delta^{-}\cdot CH^{2m+n-1}(E^{2m+n-2}, 2m+n-1) \\
\oplus\Delta\cdot\Delta^{-}\cdot CH^{2m+n-2}(E^{2m+n-2}, 2m+n-1).
\end{multline}

We now project onto the sign character eigenspace of $(\Z/2\Z)^2$.  
First of all, notice that 
\begin{eqnarray*}
\varrho^t_{0,1}\big{(}CH^0(E^2)\cdot CH^{2m+n}(E^{2m+n-2}, 2m+n-1)\big{)}&=&  \\
\big{(}\varrho^t_{0,1}(CH^0(E^2))\big{)}\cdot CH^{2m+n}(E^{2m+n-2}, 2m+n-1)&=&0.
\end{eqnarray*} Indeed, every element of $CH^0(E^2)$ is invariant under the action of $(\Z/2\Z)^2$.  Thus the alternating projection is zero.
Secondly, notice that 
$$\Delta\cdot\Delta^{-}\cdot CH^{2m+n-2}(E^{2m+n-2}, 2m+n-1)$$ also dies under this projection.
Indeed, $\Delta\cdot\Delta^{-}$ is symmetric (with respect to the action of $(\Z/2\Z)^2$), not antisymmetric.  

Since 
\begin{eqnarray*}
\text{Alt}_{(\Z/2\Z)^2}(\Delta)&=& 
\text{Alt}_{(\Z/2\Z)^2}((x,x)) \\
&=&(x,x)-(x,-x)-(-x,x)+(x,x) \\ 
%&=&\Delta-\Delta^{-}-\Delta^{-}+\Delta \\
&=&2(\Delta-\Delta^{-}),
\end{eqnarray*}
and similarly $\text{Alt}(\Delta^{-})=2(\Delta^{-}-\Delta)$, 
it follows 
that the sign character eigenspace of $(\Z/2\Z)^2$ of the
pullback of \eqref{projdecomp} is 
generated in the first two coordinates by the cycle $\Delta - \Delta^{-}$.

Finally, notice that $$[\text{Id}+\sigma_{2}](\Delta - \Delta^{-})=2(\Delta-\Delta^{-}).$$
We have shown that 

\begin{multline}
CH^{2m+n}(E^{2m+n}, 2m+n-1)\boxtimes \Sym^nh^1(E)(-m) \\
%=\varrho_{n,m}(CH^{2m+n}(E^{2m+n}, 2m+n-1))\otimes \Sym^nh^1(E)(-m) \\
%\cong\varrho_{0,1}\big(\Delta - \Delta^{-}\big)\cdot \varrho_{n,m-1}(CH^{2m+n-1}(E^{2m+n-2}, 2m+n-1))\otimes\Sym^nh^1(E)(-m) \\
=\big(\Delta - \Delta^{-}\big)\cdot CH^{2m+n-1}(E^{2m+n-2}, 2m+n-1)\boxtimes\Sym^nh^1(E)(-m).
\end{multline}

%Recalling the suppressed product term $\boxtimes\Sym^nh^1(E)(-m)$ 
The lemma now follows easily by
induction on $m$.
\end{proof}

%\end{proof}

\subsection{The Case for Cohomological Connectedness, and a Theorem}\label{hand wave}

In order to relate the computation of the previous section to the cohomology of our category, we need to justify why $\sA_E$ is quasi-isomorphic to its 1-minimal model ($\wedge(\sM_{\sA}[-1])$).  
%From the discussion in section \ref{minmod}, we first need to 
A sufficient condition to apply the algebraic topology machinery from the literature is that $\sA_E$ is connected and cohomologically connected.  As remarked in Section 1, $\sA_E$ is connected via the Adam's grading ($2m+n$).  Cohomological connectedness, however, is somewhat more subtle.  A naive approach to this question is provided by the Beilinson-Soul$\acute{\text{e}}$ conjecture.
The Beilinson-Soul\'e conjecture implies that  
$CH^{n+m}(E^n,{2m+n-i})\otimes\Q)=0$ for $2m+n-i\geq 2m+2n$, or whenever $n+i\leq 0$.  It follows that 
$\sA_E$ is cohomologically connected with respect to the ``total'' grading $-i-n$.  However, under this grading the correct Hopf algebra is difficult to see.  For instance, we can no longer take $H^0$ of the bar complex, which we can under the more natural grading by $i$, since the algebra associated to a ``total'' grading ($i+n$) is clearly not a $K(\pi, 1)$.  
%not minimal model but generalized Eilenberg-MacLane spaces?, Postnikov towers, Alg version of K(\pi_n,n+i) construction, may not be doable in general 
We will therefore proceed to refine our analysis of the cohomology of $\sA_E$ by appealing to the pure motive ``labels.''

Our philosophical approach is to think of the cohomology groups 
$$(CH^{n+m}(E^n,{2m+n-1})\otimes\Q)_{sgn}\boxtimes \Sym^nh^1(E)(-m)$$
as extensions groups 
$$\Ext^1_{\fM(E)}(\Sym^nh^1(E)(-m),\Q)\>\>(=\Ext^1_{\fM(E)}(\Q,\Sym^nh^1(E)(n+m))\>\>)$$%Ext(Q(n),Q(m))=0 for n>m (Deligne - zagier's conj)
in the category of mixed elliptic motives, and as such to map under various realizations to extension groups in various target categories. 
In other words, we are thinking of the above cohomology groups as ``cohomology with coefficients''.   More precisely, let 
$$H^i(Y,\sF):=\Ext^i(F_Y(0),\sF)$$
denote the extension groups (thought of as cohomology) in some category of mixed sheaves with coefficient ring $F$ (such as mixed $l$-adic sheaves, or even mixed motivic sheaves), where $\sF$ is some mixed sheaf on a scheme $Y$.  We apply the Leray spectral sequence to the morphism
$\pi:E^n\rightarrow \Spec(k)$.  
%$$ H^r(X,R^s\pi_*\sF)\Rightarrow H^{r+s}(Y,\sF) $$
%(where $\pi:Y\rightarrow X$ is a morphism of varieties) 
This yields a spectral sequence
$$ H^r(\Spec(k), H^s(E^n,\sF))\Rightarrow H^{r+s}(E^n,\sF). $$
%Let $CH^{n+m}:=(CH^{n+m}(E^n,{2m+n-1})\otimes\Q)$.  
The idea (for $m>0$) is to think of 
$$(CH^{n+m}(E^n,{2m+n-i})\otimes\Q)_{sgn}\boxtimes \Sym^nh^1(E)(-m)$$
as the 
$E_2$ term $$H^i(\Spec(k),\Sym^nh^1(E)(-m))$$ in the spectral sequence above.
%$i$-th associated graded of the filtration 
%$$0=F^tCH^{n+m}\subseteq \cdots F^{p+1}CH^{n+m}\subseteq F^{p}CH^{n+m}\cdots 
%\subseteq F^{1}CH^{n+m}=CH^{n+m}$$
%we get from the spectral sequence above.  
%$$ H^i_{\frak{MM}}(\Spec(k),\Sym^nh^1(E)(-m) )\subset
  Since for any cohomology theory we have $H^i(X)=0$ for $i< 0$, we expect 
$$(CH^{n+m}(E^n,{2m+n-i})\otimes\Q)_{sgn}\boxtimes \Sym^nh^1(E)(-m)=0$$
for $i<0$, or in other words, we expect that 
\begin{conj}\label{cohomconn}
$\sA_E$ is cohomologically connected with respect to $i$.
\end{conj}
This conjecture can be thought of as a strengthening of the Beilinson-Soul\'e conjecture.

%$\wedge \sM$ or $\wedge (\sM[1])$?
We now have everything we need in order to prove Proposition \ref{compprop}.  Indeed, if $\sA_E$ is cohomologically connected with respect to $i$, then by Theorem 2.30 in \cite{BK} (p.\ 567),
$\sA_E$ is quasi-isomorphic to its minimal model $\wedge\sN$.  Let $\wedge \sM:=\sM_{\sA}[-1]$ denote the 1-minimal model of $\sA_E$.  Now $\wedge \sM\subset \wedge\sN$ and $H^1(\wedge\sN)=H^1(\wedge \sM)$ since $\wedge \sM$ consists of all the elements of degree 1 of $\wedge\sN$.  In other words, the first cohomology group of the minimal model of $\sA_E$ is the same as the first cohomology group of the 1-minimal model of $\sA_E$.
Finally notice that $\varrho^t_{n,0}$ induces the sign-character eigenspace.  
Hence Proposition \ref{compprop} follows from Lemma \ref{complem}.

% $$ H^r_{abs}(\Spec(k), (R^s\pi_*\Q)(n+m))\Rightarrow H^{r+s}_{abs}(E^n, \Q)(n+m))$$.  
In fact, we expect a stronger conjecture to be true.  The minimal model machinery makes the most sense philosophically if 
\begin{conj}\label{kpi1}
$\sA_E$ is a $K(\pi,1)$ (in the sense of Sullivan, \cite{Su}, discussion beginning p.\ 316). 
\end{conj}

In other words, we expect $\sA_E$ to be quasi-isomorphic to its 1-minimal model.  
Notice that conjecture \ref{kpi1} implies conjecture \ref{cohomconn}. 
%-may not be true actually.  perhaps only need cohom connected of some sort.

Furthermore, the interested reader has already noticed, no doubt, that the computations in lemma \ref{complem} also apply to a computation of the other cohomology groups of $\sA_E$.  (In the case of a number field, however, we expect the higher extension groups to be zero.)

%if we apply conjecture \ref{kpi1} and conjecture \ref{cohomconn} to 
%Theorem \ref{igorrocks}, we get the following 

We can restate the results of this section as follows:

%We have shown 
\begin{thm}\label{BSconj=>}  If 
%our DGA $\sA_E$ is cohomologically connected with respect to $i$, and if 
$\sA_E$ is
a $K(\pi,1)$ (in the sense of Sullivan), then the cohomology of our category $\fM(E)$ agrees with the expected 
motivic cohomology groups.
\end{thm}
%%%*** Cite this label in the introduction statement of the theorem!!!!!!!!!!!!!!!!

Note that Theorem \ref{BSconj=>} is just a restatement of Theorem \ref{cohomology result} .

\begin{remark}  Our construction of $\sM$ is well defined even if $\sA_E$ does not satisfy the above conjectures.  Therefore, 
the cohomology of our category $\fM(E)$ makes sense independently of any conjectures.  
%The definition of $\fM(E)$ makes the most sense philosophically, however, if we assume the above conjectures.
\end{remark} 

%for a number field, the higher exts vanish.  Thus do we need to assume K(\pi,1)-ness, or is it sufficient to assume only cohom connected?  If Ext^1 for the 1-minimal model is iso to Ext^1 for the whole minimal model, then for a number field I don't care if the chow groups vanish or not, I WANT to take the 1-minimal model.

%*************Section 3*********************

\section{Nontrivial Elements of the Hopf
Algebra $\sH_E$}\label{chap:elts}
We will now construct families of cycle classes in each 
$\Sym^nh^1(E)(-1)$-graded piece of the
Hopf algebra $\sH_E=H^0(B(\sA_E^{\bullet}))$. 
These classes will be written out explicitly using a (non-canonical) set of choices for the projectors involved.  Note that the construction naturally lives in the subcategory $\fM(E)^{\eff}$ of effective mixed elliptic motives (subsection \ref{effective}).  
%We will write these classes out explicitly using choices we make for projectors.  
%It will follow from the nontriviality
%of such elements that our category is nontrivial.  

%An analysis of the realization functor to the category of
%variations of mixed Hodge structures should show that
%these cycles define entries in the period
%matrices associated to objects in the subcategory of
%motives which map under the regulator to the elliptic
%polylogarithmic variations of mixed Hodge structure.  

\subsection{Some Functions}\label{functions}

Let $0\in E(k)$ denote the zero element under the addition law
on the elliptic curve.  For $1\leq i\leq n$ let $p_i:E^n\rightarrow E$ denote projection to 
the $i$th component of $E^n$.  Let $p_{n+1}:=-\sum_{i=1}^{n}p_i$.
Define the divisors 
$${\bar D}_i^{(n)} = p_i^*( 0) ,\quad i = 1,2,\ldots ,n+1,$$  
$$\Delta_{i,j}^{(n)} = 
\{P\in{E} ^{n}\}|p_i(P) =p_j(P)\} ,\quad
i,j =1,2,\ldots,n+1,i\ne j .$$ 
Now define ${\bar F}_{n}(x,y_1,...,y_n)\in k(E^{n})$, $n\geq 2$, by the following 
divisor.  

\begin{equation}
\nonumber ({\bar F}_{n})= -(n)\sum_{i=1}^{n} {\bar D}_i^{(n)} + 
\sum_{1\leq i<j\leq n } 
\Delta_{i,j}^{(n)} + {\bar D}_{n+1}^{(n)} \label{Fn}
\end{equation}

For example, ${\bar F}_2(x,y)\in k(E^2)$ is defined by the divisor 
\begin{equation}\label{F2} ({\bar F}_2) = -2\{E\times\{0\}\}-2\{\{0\}\times E\}+\Delta+\Delta^{-};\qquad
{\bar F}_2(x,y)={\bar F}_2(y,x).
\end{equation} where $\Delta$ and $\Delta^{-}$ denote the diagonal and the antidiagonal on $E^2$
respectively.

 A. Levin starts with a choice of functions ${\bar F}_{n}, n\geq 2$ 
in \cite{L}, which he defines using a Vandermonde determinant of
copies of the Weierstrauss $\sP$ function, and uses associated functions to define his 
symbols in
K-theory.  It would be interesting to relate Levin's symbols to the cycle classes defined below.

\subsection{Notation for Projectors}\label{projectornotation}

In this paper we use two related constructions of all irreducible $GL(V)$-representations. 
The first standard construction is Weyl's construction of the Schur functor.  See \cite{FH}, especially Lectures 4 and 6, for example, for an introduction to the connection between projectors in $\C[\Sigma_n]$, Young Tableaux, and irreducible representations of $Gl_n$.  The second standard construction is the construction of the Spect module.  A good reference for the connection between Young Tabloids and irreducible representations of $Gl_n$ is \cite{F}, chapters 7 and 8.

%Let us explicitly compute the projection 
%$$\varrho_{a,b}: h^{a+2b}(E^{a+2b}) \surj \Sym^a\sH(b)(-a-2b).$$
Let $b=k+2l$.  
%Recall (see for example \cite{FH}, p 45) the construction of a Young symmetrizer:
Let $T$ denote either a Young tableau or a Young tabloid.  Let $R(T)$ and $C(T)$ denote the following subgroups of $\Sigma_b$:
\begin{eqnarray*}
R(T)&=&\{g\in \Sigma_b | g \text{ preserves each row of } T\} 
%\quad\quad
\\
C(T)&=&\{g\in \Sigma_b | g \text{ preserves each column of } T\}
\end{eqnarray*}
Define the following elements of the group algebra $\C\Sigma_b$:
$$c_b=\sum_{g\in R(T)}g \quad \text{ and }\quad d_b=\sum_{h\in C(T)}(\text{sgn}(h))h.$$
If $T$ is a tableau, let $\varrho_{k,l}$ denote the projector $c_b\cdot d_b\in \C\Sigma_b$; when $T$ is a tabloid, 
define the projector $\rho_{k,l}$ to be $c_b\in \C\Sigma_b$.

{\bf Notation} Let $\varrho_{n,m}$ correspond to the Young tableau
in Figure \ref{fig4}, and let $\rho_{n,m}$ correspond to the column tabloid in Figure \ref{fig9000}.

\begin{figure}[htbp]
%$$\hspace{.97in}n$$
%\vspace{-.3in}
%$$\hspace{.97in}\overbrace{\hspace{.97in}}$$
%\vspace{-.3in}
\begin{center}
\includegraphics[scale=.65]{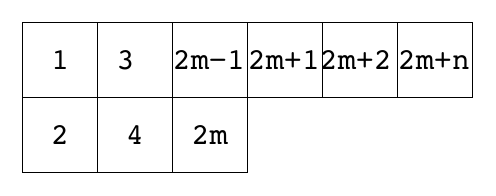}
\end{center}
\vspace{-.3in}
%$$\underbrace{\hspace{.97in}}\hspace{.97in}$$
%$$m\hspace{.97in}$$
\caption{Young Tableau for $\varrho_{n,m}$}
\label{fig4}
\end{figure}
%$$\text{Fig. 4: Young Tableau for }\varphi_{n,m}$$
\begin{figure}[htpb]
%$$\hspace{1.5in}n$$
%\vspace{-.3in}
%$$\hspace{1.5in}\overbrace{\hspace{1.35in}}$$
%\vspace{-.3in}
\begin{center}
\includegraphics[scale=.65]{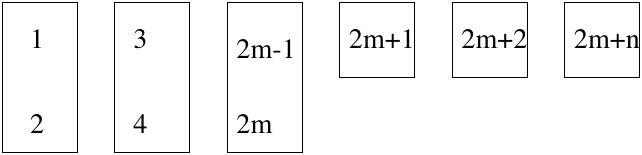}
\end{center}
%\vspace{-.3in}
%$$\underbrace{\hspace{1.35in}}\hspace{1.5in}$$
%\vspace{-.3in}
%$$m\hspace{1.5in}$$
\caption{Column Tabloid for $\rho_{n,m}$}
\label{fig9000}
\end{figure}
%$$\text{Fig. 9000: Column Tabloid for }\phi_{n,m}$$

Recall that the transpose $\varrho^t$ of a Young symmetrizer is defined by   
switching the roles of $R(T)$ and $C(T)$ in the construction of the element in the group algebra.   
%Note that $\varphi^t=\varphi\otimes \text{Alt}_{(\Z/2\Z)^{b}}$. (the representations are.)
Pictorially, the map $\varrho \mapsto \varrho^{t}$ corresponds to flipping a Young diagram about its diagonal 
%\begin{figure}[htbp]
%\begin{center}
%\includegraphics[scale=.65]{transyoung.eps}
%\end{center}
%\caption{Conjugating a Young diagram}
%\label{fig5}
%\end{figure}
%$$\text{Fig. 5: Conjugating a Young diagram}$$
while keeping track of the inscribed tableau.  
%(See Figure \ref{fig5}.)

%Note that $\varphi^t$ in general will act on a vector space $V^{\otimes p}$ where $V$ has dimension larger than 2.  ($\wedge^kV\neq 0$.)

Note that a different choice of tableau will result in different projectors thought of as elements in the group ring $\Q[\Sigma_n]$.  

Similarly, we define the transpose of a column tabloid to be the row tabloid associated to the transpose of a representative Young tableau.

It is true that two projectors with the same Young diagram but different tableau will determine isomorphic $GL_2$-representations.  However, we wish to apply the projectors to arbitrary vector spaces (in our case to algebraic cycle groups), so in our context we will need to keep track of the choice of projector.  
Note that a different choice of projector will necessitate different choices of cycles.  When we do compute with algebraic cycles, we will assume, unless otherwise stated, that $\Sym^kh^1(E)(-l)$ is given 
either by the projector $\varrho_{k,l}$ or $\rho_{k,l}$ (where the choice will be clear from context.)

%If $\phi_T$ is the Young symmetrizer associated to $T$, we abuse notation and let $[\phi_T]\>$ denote the corresponding equivalence class of Young symmetrizers.  

%Let $N^{V^{|\lambda|}}$ denote the $\Q$-vector space generated by all representations 
%$\phi_T(V^{|\lambda|})$ where $T$ runs through all numberings of the Young diagram associated to $\lambda$.  The relations $R^{\lambda}$ correspond to relations $R^{V^{|\lambda|}}\subset N^{V^{|\lambda|}}$.  We abuse notation and define 
%$[\phi_T(V^{|\lambda|})]\> \in \> N^{V^{|\lambda|}}/R^{V^{|\lambda|}}$ 
%to be the corresponding classes.

\subsection{$E$-motives}\label{sec:polylog}

Choose $n$ functions $g_1,\ldots,g_n \in k(E)^*$, 
whose divisors are pairwise disjointly supported.  Note that
this necessitates using at least $2n$ distinct points of $E(k)$.
Furthermore, notice that in order for the cycle below to be
defined either all divisors of the $g_i$'s must be disjoint from $\{(0)\}$
or we need to modify ${\bar F}_n$ slightly to avoid inadmissibility.
%\footnote{Note that this will not contradict lemma \ref{products}} 
A necessary condition for the cycles
below to be nontrivial is that $g_i(-x)\ne g_i(x)$ for all functions $g_i$.
%should not be even.
%\footnote{$g_i$ is said to be even if $g_i(-x)=g_i(x)$} 
%and odd if $g_i(-x)=g^{-1}_i(x)$}.  

For future reference:  Let $u,v,u+v\in E(k)$ denote the nonzero closed points of order two.  Choose functions $h_n\in k(E),n\in \Z^+,n>1$ that have the following divisors 
$$(h_n):= \Big{\{} \mycases{n(u)-n(0)\qquad\qquad\qquad \text{ if }n \text{ is even}}
{(n-2)(u)+(v)+(u+v)-n(0)\>\>\>\> \text{ if }n \text{ is odd}}$$
Define 
$$F_n(z_1,\dots,z_n):={\bar F}_n(z_1,\dots,z_n)h^{-1}_n(z_2)\cdots h^{-1}_n(z_n)$$

In other words, define divisors

$${ D}_i^{(n)}(u) = p_i^{(n)}{}^*(u ) ,\quad i = 1,2,\ldots ,n,$$ 
Then 
\begin{equation}
\nonumber ({ F}_{n})= -(n)\sum_{i=1}^{n} { D}_i^{(n)}(u) + 
\sum_{1\leq i<j\leq n } 
\Delta_{i,j}^{(n)} + { D}_{n+1}^{(n)}(0) 
\end{equation}
for $n$ even and 
\begin{eqnarray*}
\nonumber ({ F}_{n})&=& -(n-2)\sum_{i=1}^{n} { D}_i^{(n)}(u) - \sum_{i=1}^{n} { D}_i^{(n)}(v)-\sum_{i=1}^{n} { D}_i^{(n)}(u+v) \\
&&+
\sum_{1\leq i<j\leq n } 
\Delta_{i,j}^{(n)} + { D}_{n+1}^{(n)}(0) 
\end{eqnarray*}
for $n$ odd.

For example, $ F_2(x,y)\in k(E^2)$ is defined by the divisor 
\begin{equation}\nonumber(F_2) 
= -2\{E\times\{v\}\}-2\{\{u+v\}\times E\}+\Delta+\Delta^{-}
\end{equation}
where $\Delta=\{(x,x)|x\in E(k)\}$ denotes the diagonal on $E\times E$, and 
$\Delta^{-}=\{(x,-x)|x\in E(k)\}$ denotes the antidiagonal.

Define cycles $X^{a_1,\dots,a_r}_{F_{n+1+r},g_1,\ldots, g_n}$, $Y^a_{g_1,\ldots,g_n}$, and 
${}_jZ^{b_1,b_2}_{g_1,\ldots,g_n}$ parametrically as follows:
%Let $X_{F_{n+1+r},g_1,\ldots, g_n}$ denote the following cycle:
\begin{multline}
X^{a_1,\dots,a_r}_{F_{n+1+r},g_1,\ldots, g_n}= 
\Big{\{}\Big{(}x, (-x-\sum_{i=1}^n y_i-\sum_{j=1}^ra_j) ,y_1,...,y_{n}, \\
F_{n+1+r}(x,y_1,...,y_n,a_1,\dots,a_r),g_1(y_1),...,g_n(y_n)\Big{)} \\
\big{|}\>\> 
(x,y_1,...,y_{n})\in E^{n+1}, a_j\in E(k), 1\leq j\leq r\Big{\}} \\
\in
\sZ^{n+2}(E^{n+2},n+1))
\end{multline} 
%Let $Y_{g_1,\ldots,g_n}$ denote the following cycle:
\begin{multline}
Y^a_{g_1,\ldots,g_n}:= 
\Big{\{}((-\sum_{i=1}^n y_i-a) ,y_1,...,y_{n}, 
g_1(y_1),...,g_n(y_n)) \\
 \>\>   \Big{|}\>\> 
(y_1,...,y_{n})\in E^{n}, a\in E(k), 1\leq j\leq p\Big{\}} \\
\in
\sZ^{n+1}(E^{n+1},n))
\end{multline}
\begin{multline}
{}_jZ^{b_1,b_2}_{g_1,\ldots,g_n}:= \\
\Big{\{}(x,(-x-\sum_{i=1,i\neq j}^n y_i-b_1-b_2) ,y_1,\dots,\hat{y_j},\dots,,y_{n}, g_1(y_1),\dots,g_j(b_2),\dots,g_n(y_n))
\\    \big{|}\>\> 
(x,y_1,\dots,\hat{y_j},\dots,y_{n})\in E^{n} , \>\> b_1,b_2\in E(k), 1\leq j\leq n\Big{\}} \\
\in
\sZ^{n+1}(E^{n+1},n))
\end{multline}
where $\hat{g_j}$ or $\hat{y_i}$ denotes omission. 

% This is an element in the vector space spanned by the cycle.  
% is summation of elements as cycles, but with label attached.  label shows where element lives.
%Let $\langle A\rangle$ denote the $\Q$-vector space spanned by A.
%Let $\langle A\rangle$ denote the class represented by A after.
%%independent of the rational multiple.  
%%denote that
%denote a possibly $\Q$-rational multiple of the cycle defined by A.
\begin{defn}  We have
\begin{eqnarray*}
%\begin{enumerate}
%\begin{multline}\nonumber
%\label{moregeneralcycle}
%\eta_{\Sym^{n}h^1(E)}^{a_1,\dots,a_r} 
\eta_{\Sym^{n}h^1(E)(-1)}^{a_1,\dots,a_r}(g_1,\ldots,g_n)) 
&:= &
%\\
%\!\!\!\!\!\!\!\!\!\!\!\!\!\!\!\!\!\!\!\!\!
%this'll work too (lots more of them though) \!=\negthinspace 
%\negthickspace\negthickspace\negthickspace\negthickspace
%also there is \negmedspace
%\{
\langle\rho^t_{n,1}(X^{a_1,\dots,a_r}_{F_{n+1+r},g_1,\ldots, g_n})\rangle\boxtimes 
%S\sH_{n,-n-1} 
\rho_{n,1}\cdot(h^1(E))^{\otimes n+2} \\
&&\subset
\sZ^{n+2}(E^{n+2},n+1)
\boxtimes\Sym^{n}h^1(E)(-1) \\
%\end{multline} 
%\end{defn}
%\begin{defn}
%\begin{multline}\nonumber
%$$
%\!\!\!\!\!\!\!\!\!\!\!\!\!\!\!\!\!\!\!\!\!
(\text{In particular, } \eta_{h^1(E)}(p) 
%\>\>\>\>\>\>\>\>\>\>\>\>\>\>
&:= &\langle\big{(}(p)-(-p)\big{)}\rangle\otimes h^1(E) 
\subset \sZ^1(E)\boxtimes h^1(E))\\
%\end{multline} 
%\end{defn}
%\eta_{\Q(-1)}=(\eta_{\Q(-1)}(a))=
%\{\rho_{0,1}((x, -x,f_a(x)))\otimes\Q(-1)|x\in E^{n+1} \} \\
%\begin{defn}
%\item 
%\begin{multline}\nonumber
\mu_{\Sym^{n+1}h^1(E)}^{a}(g_1,\ldots,g_n)) &:= &
\langle Y^a_{g_1,\ldots,g_n}\rangle\boxtimes  \Sym^{n+1}h^1(E)
\\ &&\subset
\sZ^{n+1}(E^{n+1},n)
\boxtimes \Sym^{n+1}h^1(E) \\
%h^1(E)^{\otimes n+1}
%\end{multline} 
%\end{defn}
%\begin{defn}
%\item 
%\begin{multline}\nonumber
 {}_j\nu_{\Sym^{n-1}h^1(E)(-1)}^{b_1,b_2}(g_1,\ldots,g_n)) &:= &
%\!\!\!\!\!\!\!\!\!\!
\langle\rho_{n-1,1}^t({}_jZ^{b_1,b_2}_{g_1,\ldots,g_n})\rangle\boxtimes 
\rho_{n-1,1}\cdot h^1(E)^{\otimes n+1}
%\Sym^{n-1}h^1(E)
\\&& \subset
\sZ^{n+1}(E^{n+1},n)
\boxtimes\Sym^{n-1}h^1(E)(-1)
\end{eqnarray*}
%\end{multline}
%\end{enumerate}
\end{defn}

\begin{remark} Since 
$${\bar F}_n|_{p_i^{(n)}{}^*( 0)}={\bar F}_{n-1}\text{ for } 1\leq i\leq n,$$
the following cycle is well defined:
%\begin{small} 
%\begin{multline}\nonumber
$$ \eta_{\Sym^{n}h^1(E)(-1)} := \eta_{\Sym^{n}h^1(E)(-1)}^{0,\dots,0}(g_1,\ldots,g_n)). $$
%\end{multline}
\end{remark}

When one unwinds the various definitions one finds that

%\begin{prop}
%\begin{multline} 
\begin{eqnarray}
\partial \eta_{\Q(-1)}^{a_1,\ldots,a_r}     
&=&\>\>\sum_{r}\>\>
\delta\>[\eta_{h^1(E)}(a_i)\otimes\eta_{h^1(E)}(-a_i-\sum_{j=1}^r a_j)].
\end{eqnarray}
%\end{multline}
%\end{prop}
and
%\begin{prop}
\begin{multline}
\partial \eta_{\Sym^nh^1(E)(-1)}^{a_1,\ldots,a_r}(g_1,...,g_n)    \\ 
%\sum_i \sum_{p\in div(g_i)}
%\Gamma_{F,g_1,...,{\hat{g_i}},...,g_n}|_{{E}^{(n)}_p}
%=\sum_i \sum_{p\in div(g_i)}
%\delta[\Gamma_{F_{n},\vec{g}_{j\neq i}}^p\otimes(\eta_{h^1(E)(-1)}(p))] 
=\delta\>\Big{[}\sum_i \sum_{p\in div(g_i)}
\eta_{\Sym^{n-1}h^1(E)(-1)}^{a_1,\ldots,a_r,p}(g_1,...,{\hat{g_i}},...,g_n)\otimes(\eta_{h^1(E)}(p))  \\
+\sum_{l=1}^r\eta_{h^1(E)}(a_l)\otimes\mu_{\Sym^{n}h^1(E)}^{a_1+\dots+2a_l+\dots+a_r}(g_1,\cdots, g_n) \\
+\sum_l\sum_j{}_j\nu_{\Sym^{n-1}h^1(E)(-1)}^{a_1+\dots+a_r,a_l}(g_1,\ldots,g_n))\otimes\eta_{h^1(E)}(a_l)\Big{]}.
\end{multline}
%\end{prop}

Also note that 

\begin{multline} 
\partial(\mu_{\Sym^{n}h^1(E)}^{a}(g_1,\cdots, g_n)) \\
=\sum_i \sum_{p\in div(g_i)}
\delta[\mu_{\Sym^{n-1}h^1(E)}^{a+p}(g_1,...,{\hat{g_i}},...,g_n)\otimes(\eta_{h^1(E)}(p))]
\end{multline}

Furthermore, 
%Note that this is as complicated as it gets.  Indeed, 
\begin{eqnarray*}
\partial({}_j\nu_{\Sym^{n-1}h^1(E)(-1)}^{a,b}(g_1,\ldots,g_n)))&=&0,
\end{eqnarray*} 
since the boundary of this cycle dies under the alternating condition on the 
\linebreak $\P^1\setminus\{1\}$ coordinates.  Indeed, let $\sigma_{i,j}$ be the transposition that switches the $i$th and $j$th $\P^1\setminus\{1\}$-coordinate.  Then 
$$\partial({}_jZ^{a,b}_{g_1,\ldots,g_n}-\sigma_{i,j}({}_jZ^{a,b}_{g_1,\ldots,g_n}))=0.$$ 
%In particular, the terms ${}_j\nu_{\Sym^{n-1}h^1(E)}^{a,a}$ have no boundary. 
%Furthermore,

%\end{prop}

Thus we have shown
\begin{prop}\label{thering}
$\partial C$ of an element $C$ in the subring $R$ generated by the 
\linebreak
$\eta_{\Sym^{n}h^1(E)(-1)}^{a_1,\dots,a_r} $'s, the $\mu_{\Sym^{n+1}h^1(E)}^{a}$'s, and the 
${}_j\nu_{\Sym^{n-1}h^1(E)(-1)}^{b_1,b_2}$'s (where $n\geq 0$ and $a,b_1,b_2,a_1,\dots,a_r$ are closed points of $E$) is the image under the product map of a linear combination of cycles again in $R$. 
\end{prop} 
%\begin{remark} It would be interesting to work out all the relations in $R$.
%\end{remark}

However, the situation is much simpler that it first appears.

\begin{prop}\label{boundaryterms}
The $\mu_{h^1(E)^{n+1}(-n-1)}^{a}(g_1,\ldots,g_n))$ and ${}_j\nu_{\Sym^{n-1}h^1(E)}^{b_1,b_2}(g_1,\ldots,g_n))$ terms appearing in the ``successive boundary'' terms of $\eta_{\Sym^{n}h^1(E)(-1)}$ are trivial $\in H^0(B(\sA_E))$. 
\end{prop}

Sketch of proof:

We have already remarked above that the boundary of ${}_j\nu_{\Sym^{n-1}h^1(E)}^{b_1,b_2}(g_1,\ldots,g_n))$ is zero.  The terms of this cycle appearing in the successive boundaries of $\eta_{\Sym^{n}h^1(E)(-1)}$ are killed by cycles of the form
\begin{eqnarray*}
\Big{\{}(x,(-x-\sum_{i=1}^n y_i-b_1) ,y_1,\dots,\hat{y_j},\dots,,y_{n}, g_1(y_1),\dots,g_j(y_j),\dots,g_n(y_n), g_j(b_2))\Big{\}}
\\
\end{eqnarray*}
where $b_2$ is in the divisor of $g_j$.  

For example, $\sum_{a\in (g_1)}(y,-y-a,g_2(b))$ is killed by $(y,y-z,g_1(z),g_2(b))$.  
%Here, $a$ is in the divisor of $g$.  

The case of $\mu_{h^1(E)^{n+1}(-n-1)}^{a}(g_1,\ldots,g_n))$ is a bit more involved.  One shows that the final term of its successive boundaries is zero, (hence the class defined by such cycles does actually go to zero) and then one can kill it with terms of the form 
\begin{eqnarray*}
\Big{\{}((-\sum_{i=1}^n y_i-z) ,y_1,...,y_{n}, 
g_1(y_1),...,g_n(y_n), g_i(z)) \\
 \>\>   \Big{|}\>\> 
(y_1,...,y_{n})\in E^{n},  1\leq i\leq n\Big{\}} \\
\end{eqnarray*}

For example, the cycle $(x,-x-y,g(x), g(y))$ and its successive boundary kills the cycle $\sum_{a\in(g)}(x,-x-a, g(x))$ and its successive boundary.
%appearing in the successive boundaries of $\eta_{\Sym^{2}h^1(E)(-1)}$.
%As before, $a$ is in the divisor of $g$.  

End of sketch of proof.

Hence we can use 
$\eta_{\Sym^{n}h^1(E)(-1)}$ and its ``successive boundaries'' 
%(in the sense of \eqref{firsttry}) 
to define a cohomology class
$\in H^0(B(\sA_E))$.  
Explicitly, 
\begin{multline} \label{class?}
\sE(g_1,...,g_n):= 
{\Big\{}\eta_{\Sym^nh^1(E)(-1)}(g_1,...,g_n),    \\ 
\sum_i \sum_{p\in div(g_i)}
[\eta_{\Sym^{n-1}h^1(E)(-1)}^p(g_1,...,{\hat{g_i}},...,g_n)\otimes(\eta_{h^1(E)}(p))], \\
%\sum_j\sum_{q\in div(g_j)}\sum_i \sum_{p\in div(g_i)}
%[\eta_{\Sym^{n-2}h^1(E)(+1)}^{p,q}(g_1,\dots,{\hat{g_i}},\dots,{\hat{g_j}},\dots,g_n)\otimes((q)-(-q))]\otimes((p)-(-p))], \\
,\dots, \\
%\sum [\eta_{\Q(-1)}^{a_1,\dots, a_n}
%\otimes((a_n)-(-a_n))\cdots\otimes((a_1)-(-a_1))], \\
\sum \bigotimes^{n+2}_{j=1}\eta_{h^1(E)}(b_j){\Big\}}
\end{multline}
defines a family of cohomology classes
$\in H^0(B(\sA_E))$.  Note: the $b_j$'s and the limits for the sum in the final term are completely determined by the divisors of the $g_i$'s.

 Define 
\begin{multline} \label{generalclass?} 
\sE^{a_1,\ldots,a_r}(g_1,...,g_n):=
{\Big\{}\eta_{\Sym^nh^1(E)(-1)}^{a_1,\ldots,a_r}(g_1,...,g_n), \\
\Big{[}\sum_i \sum_{p\in \text{div}(g_i)}
[\eta_{\Sym^{n-1}h^1(E)(-1)}^{a_1,\ldots,a_r,p}(g_1,...,{\hat{g_i}},...,g_n)\otimes(\eta_{h^1(E)}(p))] \\
%+\sum_{l=1}^r\eta_{h^1(E)}(a_l)\otimes\mu_{\Sym^{n}h^1(E)}^{a_1+\dots+2a_l+\dots+a_r}(g_1,\cdots, g_n) \\
%+\sum_l\sum_j{}_j\nu_{\Sym^{n-1}h^1(E)(-1)}^{a_l,a_1+\dots+a_r}(g_1,\ldots,g_n))\otimes\eta_{h^1(E)}(a_l)\Big{]}\\
,\dots, \\
\sum \bigotimes^{n+2}_{j=1}\eta_{h^1(E)}(c_j){\Big\}},
\end{multline}
%
%Furthermore, if we look at the subring $\sN$ generated by $\eta_{\Sym^nh^1(E)(-1)}^{a_1,\ldots,a_r}(g_1,...,g_n)$ and its successive boundaries,
% 
Let $[p]$ denote the class in $H^0(B(\sA_E))$ defined by $\eta_{h^1(E)}(p)$.  It follows from our explicit calculations that 

%*************Check this when proofreading**************
\begin{multline}
\psi(\sE(g_1,...,g_n))={\Big\{}\sE(g_1,...,g_n)\otimes 1 {\Big\}}+\\
{\Big\{}\sum_i \sum_{p\in \text{div}(g_i)}\sE^p(g_1,\dots,\hat{g_i},\dots,g_n)\otimes [p]{\Big\}}+\\
{\Big\{}\sum_i \sum_{p\in \text{div}(g_i)}\sum_j \sum_{q\in \text{div}(g_j)}\sE^{p,q}(g_1,\dots,\hat{g_i},\dots,\hat{g_j},\dots,g_n)\otimes ([p]\otimes [q])+\\
%\sum_i \sum_{p\in \text{div}(g_i)}\sF^{2p}(g_1,\ldots,\hat{g_i},\dots,g_n))\otimes ([p]\otimes [p])+\\
%\sum_i \sum_{p\in \text{div}(g_i)}\sum_j[{}_j\nu_{\Sym^{n-1}h^1(E)(-1)}^{p,p}(g_1,\ldots,\hat{g_i},\dots,g_n))]\otimes ([p]\otimes [p])+{\Big\}}\\
,\dots, \\
+{\Big\{}1\otimes \sE(g_1,...,g_n){\Big\}}
\end{multline} 
where $\hat{}$ denotes omission and where $\psi$ denotes the co-multiplication map
$$\psi: H^0(B(\sA))\to H^0(B(\sA))\otimes H^0(B(\sA)).$$ 
Hence $\eta_{\Sym^{n}h^1(E)(-1)}$, ${\big\{}\eta_{\Sym^{n-1}h^1(E)(-1)}^p(g_1,...,{\hat{g_i}},...,g_n)| 1\leq i\leq n \text{ and } p\in \text{div}(g_i){\big \}}$
,$\ldots$,${\big \{}\eta_{h^1(E)}(p)\> |\> p\in \text{div}(g_i)\text{ for some }i{\big \}}$and $1$ span a $H^0(B(\sA_E))$-comodule, and hence define a motive.

% (I can do this explicitly if necessary)

%It follows from the above description of the boundary that 

%\begin{small}
%\begin{multline}
%\{\eta_{\Sym^{n}h^1(E)(-1)}, \\
%\sum_i \sum_{p\in \text{div}(g_i)}
%[\eta_{\Sym^{n-1}h^1(E)}^p(g_1,...,{\hat{g_i}},...,g_n)\otimes((p)-(-p))],
% \\
%\sum_{i\neq j} \sum_{p\in \text{div}(g_i)}\sum_j \sum_{q\in \text{div}(g_j)}
%[\eta_{\Sym^{n-2}h^1(E)(+1)}^{p,q}(g_1,...,{\hat{g_i}},...{\hat{g_j}},...,g_n)\otimes((q)-(-q))]\otimes((p)-(-p)), \\
%\dots, 
%\sum_{a_i\in \text{div}(g_i)}\{ ((a_1)-(-a_1))\otimes\dots\otimes ((a_n)-(-a_n))\otimes ((\sum a_i)-(-\sum a_i))^{\otimes 2} \} 
%%(g_1)*(g_2)*\cdots *(g_n)
%\end{multline}
%\end{small}
%defines a cohomology class
%$\in H^0(B(\sA_E))$.

Finally, notice that, for any vector space $V$, 
\linebreak $\varrho^t_{n,1}\cdot\rho^t_{n,1}(V^{\otimes n+2})=\varrho^t_{n,1}(V^{\otimes n+2})$.   In particular, if $\rho^t_{n,1}(V^{\otimes n+2})\neq 0$, then \linebreak
$\varrho^t_{n,1}\cdot\rho^t_{n,1}(V^{\otimes n+2})\neq 0$.  Hence if 
$\sE^{a_1,\ldots,a_r}(g_1,...,g_n)$ defines a nontrivial class in 
$\sHf$, then its $\Sym^{n}h^1(E)(-1)$-projection also 
%\begin{multline}
%\Hom_{GL(h^1(E))}\big{(}\>\>
%\Sym^{n}h^1(E)(-1)\>\>,\>\>\sE^{a_1,\ldots,a_r}(g_1,...,g_n)\>\>\big{)}
%\\ =\Hom_{GL(h^1(E))}\big{(}\>\>
%\varrho^t_{n,1}\cdot h^{\otimes{n+2}}\>\>,\>\>\sE^{a_1,\ldots,a_r}(g_1,...,g_n)\>\>\big{)}
%\end{multline}
defines a nontrivial class in $\sHf$.  
Thus a suitable linear combination of such classes 
%(and in general some assumptions about the Beilinson regulator), 
will determine elements in 
%$\Ker(\bar{\text{d}})$ (where 
\begin{multline}
\Ker\Big{(}\sM_{\Sym^nh^1(E)(-1)}\otimes \Sym^nh^1(E)(-1)\stackrel{\bar{\text{d}}}{\longrightarrow
} \\
\sM_{\Sym^{n-1}h^1(E)(-1)}\otimes \Sym^{n-1}h^1(E)(-1)\otimes \sM_{h^1(E)}\otimes h^1(E)
%\sum_{\Sym(\sV\otimes\sW)} (\sM_{\sV}\bigwedge \sM_{\sW})\otimes (\Sym(\sV\otimes\sW))
%\bigoplus \sum_{\sV\wedge\sW}(\Sym(\sM_{\sV}\otimes\sM_{\sW}))\otimes (\sV\wedge\sW)
%\\
%\big{(}\sM_{\sV}\bigwedge\sM_{\sW}\big{)}\otimes \big{(}\sV\wedge \sW \big{)}
\longrightarrow \cdots\Big{)}
\end{multline}

%In the next section we will give a 
%Given a 
%suitable description of the realization functor to the category of mixed Hodge structures.  As a consequence 
Given a suitable description of a realization functor to the category of mixed Hodge structures,
we expect such elements to determine 
nonzero multiples of
$L(\Sym^nE,n+1)$.

%********************Nontriviality *****************

Notice that the final term $\sum \bigotimes^{n+2}_{j=1}\eta_{h^1(E)}(c_j)$ of $\sE(g_1,...,g_n)$ is generically not a coboundary, since, for a point $P\in E(k), (P)-(-P)$ is not the divisor of a function.  It follows therefore that the element $\sE(g_1,...,g_n)$ itself is nontrivial.  Thus we have shown 
%A simple consequence of Proposition \ref{boundaryterms} and an explicit computation of the successive boundaries of $\eta_{\Sym^{n}h^1(E)(-1)}$ is the following 

%\begin{remark} A simple computation of the first few examples yields the following result:
\begin{thm} \label{nontriviality} $\sM_{\sU}\neq(0)$ for pure motives $\sU=\Sym^{n}h^1(E)(-1)$, $n$ any natural number.
%for pure motives $\sU=h^1(E)$,$\Q(-1)$, $h^1(E)(-1)$, and $\Sym^{2}h^1(E)(-1)$.  
%In particular, our category is nontrivial!
%The category ${\mathfrak P}$ is nontrivial
\end{thm}

In particular, our category is nontrivial!  

%We will give a different proof of this result in the following section on the realization functor to the category of mixed elliptic Hodge structures.

\begin{remark} Theorem \ref{nontriviality} does not immediately lead to special values of $L$-functions $L(\Sym^nE,s)$ associated to $E$.  However, we expect that, under an appropriate mixed Hodge realization functor, certain linear combinations of the images of the above elements will 
determine (up to standard factors) positive integral values $L(\Sym^nE,n+1)$ of said $L$-function.  We hope to address this question in a future paper (\cite{P2}).
\end{remark}

%\begin{conj} The cohomology class $\sE^{a_1,\ldots,a_r}(g_1,...,g_n)$ 
%defined by $\eta_{\Sym^{n}h^1(E)(-1)}$ 
%is nontrivial for a generic choice of $g_1,\ldots,g_n$.
%\end{conj}
%\end{remark}

\vspace{.1in}

%\begin{quotation}
%\begin{center}
{\bf{Acknowledgements}}
%\end{center}
%\begin{ack}
I am indebted to my dissertation advisor Spencer Bloch, whose generosity of ideas, generosity of time, and infectious enthusiasm for mathematics in general and motives in particular, has been and will always be an inspiration to me.  I thank Sasha Beilinson, Sasha Goncharov, Igor Kriz, Andr\'e Levin, Madhav Nori, and Joerg Wildeshaus both for helpful conversations and for the ideas they presented in their papers and/or preprints. 
%(mostly listed in the Bibliography).  
I also thank the following people for helpful conversations: Vladimir Baranovsky, Prakash Belkale, Patrick Brosnan, Joe Chuang, Rob DeJeu, Herbert Gangl, Brendan Hassett, Ken Kimura, Paul Li, Peter May, Mike Mandell, Stephan Muller-Stach, Laura Scull, Ramesh Sreekantan, Burt Totaro, Charles Weibel, and the participants in the AMS special session on K-theory at the University of Wisconsin-Milwaukee (October 1997).  I thank Brendan Hassett, Kenichiro Kimura, Vladimir Baranovsky, Andr\'e Levin, Marc Levine, Amnon Besser, Jakob Scholbach, and several anonymous referrees for reading preliminary versions of this paper and making helpful comments.  (Any errors that remain are of course mine alone).  
\bibliographystyle{plain}

\end{document}